\documentclass[12pt]{article}
\usepackage{amssymb, amsmath, hyperref}
\usepackage{amsthm,fixltx2e}
\usepackage[margin=1in]{geometry}
\usepackage{enumerate}
\usepackage{bm}
\allowdisplaybreaks[1]
\numberwithin{equation}{section}

\newcommand{\ccombinedarroww}{\stackrel{\rightarrow}{\smash{\sim}\rule{0pt}{0.3ex}}}

\newcommand{\ccombinedarrowss}{\stackrel{\leftrightarrows}{\smash{\sim}\rule{0pt}{0.25ex}}}

\newtheorem{example}{Example}[section]
\newtheorem{thm}{Theorem}[section]
\newtheorem{cor}{Corollary}[section]
\newtheorem{note}{Note}[section]
\newtheorem{pro}{Proposition}[section]

\newtheorem{defn}{Definition}[section]

\newtheorem{lemma}{Lemma}[section]
\newtheorem{observation}{Observation}[section]
\usepackage{graphicx}
\begin{document}
	\vspace{-2cm}
	\markboth{G.Kalaivani and R. Rajkumar}{New matrices for the spectral theory of mixed graphs, Part I}
	\title{\LARGE\bf New matrices for the spectral theory of mixed graphs, part I}
	\author{G. Kalaivani\footnote{e-mail: {\tt kalaivani.slg@gmail.com}}~\footnote{G. Kalaivani is thankful to the University Grants Commission (UGC), Government of India, for the award of UGC-SRF  (NTA Ref. No.: 211610155238)},\ \ \
		R. Rajkumar\footnote{e-mail: {\tt rrajmaths@yahoo.co.in (Corresponding Author)}}\ \\
		{\footnotesize Department of Mathematics, The Gandhigram Rural Institute (Deemed to be University),}\\ \footnotesize{Gandhigram -- 624 302, Tamil Nadu, India}\\[3mm]
	}
	\date{}
	\maketitle
	\begin{abstract}
	In this paper, we introduce a matrix for a mixed graph, called the integrated adjacency matrix.   This matrix uniquely determines a mixed graph, as long as the indices of the matrix are specified. 	 Additionally, we associate an (undirected) graph with each mixed graph, enabling the spectral analysis of the integrated adjacency matrix to connect the structural properties of the mixed graph and its associated graph.  Furthermore,  we define certain mixed graph structures and establish their  relationships to the eigenvalues of the integrated adjacency matrix.\\
	
	\noindent \textbf{Keywords:} Mixed graph, Integrated adjacency matrix, Spectrum \\
	\textbf{Mathematics Subject Classification:}  05C50
	
\end{abstract}

\section{Introduction}

All graphs (directed graphs) considered in this paper may contain multiple loops and multiple edges (multiple directed loops and multiple arcs). A mixed graph generalizes both graphs and directed graphs. That is, a mixed graph may contain multiple loops, multiple edges, multiple directed loops and  multiple arcs. 

A fundamental problem in spectral theory  of mixed graphs  is assigning an appropriate matrix to a mixed graph and analyzing the mixed graph's properties through the eigenvalues of that matrix. To achieve this,  researchers have introduced various  matrices that capture the spectral properties of mixed graphs. 

For a mixed graph having no multiple edges, loops, multiple arcs, directed loops, and digons (i.e., a pair of arcs with the same end vertices but in opposite directions), the following matrices have been explored: the Hermitian adjacency matrix, independently introduced by Liu and Li~\cite{Liu2015hermitian} and Guo and Mohar~\cite{guo2017hermitian};  the Hermitian Laplacian matrix defined by Yu and Qu~\cite{Yu2015hermitian}; the Hermitian quasi-Laplacian matrix defined by Yu et al.~\cite{Yu2017singularity}; the Hermitian normalized Laplacian matrix defined by Yu et al.~\cite{Yu2019Hermitian}; the mixed adjacency matrix defined by Adiga et al.~\cite{Adiga2016mixed};  the $k$-generalized Hermitian adjacency matrix for each positive integers $k$, proposed by Yu et al.~\cite{Yu2023k}; the index weighted Hermitian adjacency matrices defined by Zheng Wang et al.~\cite{Wang2024index}; the $\alpha$-Hermitian adjacency matrix for $|\alpha|=1$ defined by Abudayah et al.~\cite{abudayah2022hermitian}; the $\gamma$-signless Laplacian adjacency matrix introduced by Alomari et al.~\cite{Alomari2022signless}; the Hermitian Laplacian matrix of second kind and the Hermitian (quasi-)Laplacian matrix of second kind introduced by Xiong et al.~\cite{Xiong2023principal}.

For a mixed graph having no directed loops, mixed Laplacian matrix was defined by Bapat et al.~\cite{Bapat1999generalized}. 
For a mixed graph having no loops, directed loops and digons, a generalization of the Hermitian adjacency matrix was introduced by Guo and Mohar~\cite{guo2017hermitian}. 
In the case of a mixed graph where an edge is treated as a digon, the Hermitian adjacency matrix of the second kind was introduced by Mohar~\cite{Mohar2020new}. 
For a mixed graph having no digons, loops and directed loops, another form of Hermitian adjacency matrix was defined by Yuan et al.~\cite{Yuan2022hermitian}.

However, there are certain limitations: (i) the matrix defined in~\cite{Adiga2016mixed} is not symmetric, and as a result, their eigenvalues can include complex numbers; (ii) some of these matrices are only defined for specific types of mixed graphs; and (iii) in some matrices, undirected edges are represented as digons. This representation makes it  unclear from the matrix entries whether two vertices are joined by an edge or a digon, even when the index set of the matrix is provided.  Furthermore, the number of edges, loops, arcs, and the orientation of arcs cannot be clearly determined from the entries of the matrix defined in~\cite{Bapat1999generalized}. Consequently, not all mixed graphs can be determined from their associated matrices.  

To address these issues, in this paper, we introduce a matrix, named the integrated adjacency matrix for a mixed graph. We study the properties of a mixed graph through the eigenvalues and characteristic polynomial of its integrated adjacency matrix.


The rest of the paper is arranged as follows. We present some preliminary definitions and notations in Section~\ref{S2 prelims} that are used throughout the paper. In Section~\ref{S3 adjacency matrix}  we define a matrix for a mixed graph, namely the integrated adjacency matrix, and associate a graph to it. We show that the integrated adjacency matrix of a mixed graph is identical to the adjacency matrix of the associated graph. 

Section~\ref{S4basic results}  presents some basic results which relate the mixed graph invariants, such as number of edges, number of arcs to the eigenvalues of its integrated adjacency matrix. Additionally, we define a new structure called an alternating walk in a mixed graph and use it to  determine the entries of the $k$-th power of the integrated adjacency matrix of a mixed graph.

In Section~\ref{S5 spectra} we determine the eigenvalues of the integrated adjacency matrix of some special mixed graphs, such as complete $k$-partite mixed graph, complete $k$-partite directed graph, and complete mixed graph. 
We also provide a necessary and sufficient condition on the eigenvalues and the eigenvectors of the integrated adjacency matrix for a mixed graph being $(r,s)$-regular. Furthermore, we derive the  characteristic polynomial of the integrated adjacency matrix of a mixed graph formed by adding a new vertex in a specific manner to a given mixed graph. 

In Section~\ref{S6 components}, we introduce a new structure in a mixed graph, namely a mixed component. We prove that the number of mixed components of a mixed graph equals the number of components of the associated graph. 

Section~\ref{S7 determinant} provides the determinant of the integrated adjacency matrix for specific types of mixed graphs. Finally, in Section~\ref{S8 bounds}, we  establish bounds for the largest,  second largest,  smallest,  and second smallest eigenvalues of the integrated adjacency matrix in relation with the mixed graph invariants, such as the number of edges, arcs, and degrees. Additionally, we derive upper bounds for the clique number and the independence number of a mixed graph using the eigenvalues of its integrated adjacency matrix.
	
	\section{Preliminaries}\label{S2 prelims}
In this section, we present some notations and definitions of graphs, directed graphs and mixed graphs. A mixed graph $G$ is an ordered $5$-tuple $G=(V_G,E_G,\vec{E}_G,\phi_G,\psi_G)$, where $V_G\neq\emptyset$, $E_G$, and $\vec{E}_G$ are three mutually disjoint sets, and $\phi_G$ (resp. $\psi_G$) is the incidence function which maps each element of $E_G$ (resp. $\vec{E}_G$) with an unordered pair (resp. ordered pair) of elements of $V_G$. The sets $V_G$, $E_G$, and $\vec{E}_G$ are called vertex set, edge set and arc set, respectively. Elements of the sets are called vertices, edges and arcs, respectively. 
Let $V_G=\{v_1,v_2,\ldots,v_n\}$.
If there is an edge $e\in E_G$ such that $\phi_G(e)=\{v_i,v_j\}$, then we say that $e$ is an edge joining $v_i$ and $v_j$ in $G$. In particular, if $\phi_G(e)=\{v_i,v_i\}$, then we say that $e$ is a loop at $v_i$. If there is an arc $a\in \vec{E}_G$ such that $\psi_G(a)=(v_i,v_j)$, then we say that $a$ is an arc from $v_i$ to $v_j$ in $G$. In this case, we say that $a$ starts at $v_i$ and ends at $v_j$. In particular, if $\psi_G(a)=(v_i,v_i)$, then we say that $a$ is a directed loop at $v_i$. 
If two or more edges joining a same pair of vertices, then they are called multiple edges.
If two or more arcs start at $u$ and end at $v$ for some $u,v\in V_G$, then they are called multiple arcs. 
In particular, if two or more loops (resp. directed loops) are at one vertex, then they are called multiple loops (resp. multiple directed loops). If there is an edge joining $v_i$ and $v_j$ in $G$, we denote it as $v_i \sim v_j$. To specify the number of edges joining $v_i$ and $v_j$, we write $v_i\overset{k}{\sim} v_j$, provided there are $k$ such edges in $G$. If $v_i\overset{l}{\sim} v_i$, this indicates that there are $l$ loops at $v_i$ in $G$. If there is an arc from $v_i$ to $v_j$ in $G$, we denote it as $v_i \rightarrow v_j$. To indicate the number of arcs from $v_i$ to $v_j$, we write $v_i\overset{k}{\rightarrow} v_j$, provided there are $k$ such arcs in $G$. If $v_i\overset{l}{\rightarrow} v_i$, this means there are $l$ directed loops at $v_i$ in $G$.	Two vertices $u$ and $v$ in $G$ are said to be \textit{adjacent} if at least one of $u\sim v$, $u\rightarrow v$ or $v\rightarrow u$ holds. An edge $e$ is said to be \textit{incident} at a vertex $u$ in $G$ if $\phi_G(e)=\{u,v\}$ for some vertex $v$ in $G$. An arc $a$ is said to be \textit{incident} at a vertex $u$ in $G$ if $\psi_G(a)=(u,v)$ or $\psi_G(a)=(v,u)$ for some vertex $v$ in $G$. $G$ is said to be loopless if it has no loops. 

The \textit{undirected degree} $d(u)$ of  a vertex $u$ in $G$ is the sum of the number of edges incident with $u$ and the number of loops at $u$, where each loop at $u$ is counted twice. The \textit{out-degree} $d^+(u)$ of $u$ in $G$ is the number of arcs start at $u$ (which includes the number of directed loops at $u$). The \textit{in-degree} $d^-(u)$ of $u$ in $G$ is the number of arcs end at $u$ (which includes the number of directed loops at $u$). $l(u)$ denotes the number of loops at $u$. $e(G),a(G)$ and $l(G)$ denote the number of edges (excluding loops), arcs (including directed loops) and loops in $G$, respectively. $G$ is said to be \textit{$(r,s)$-regular} if $d(u)=r$ and $d^+(u)=s=d^-(u)$ for all $u\in V_G$. A \textit{walk} in $G$ is a sequence $W:v_{n_1},e_1,v_{n_2},e_2,\ldots,e_{k-1},v_{n_k}$, where $v_{n_i}\in V_G$, $e_j\in E_G$ and $\phi_G(e_j)=\{v_{n_j},v_{n_{j+1}}\}$, or $e_j\in \vec{E}_G$ and $\psi_G(e_j)=(v_{n_j},v_{n_{j+1}})$ or $\psi_G(e_j)=(v_{n_{j+1}},v_{n_j})$ for $i=1,2,\ldots,k$, $j=1,2,\ldots,k-1$.
In this case, we say that $W$ is a walk from $v_{n_1}$ to $v_{n_k}$. The total number of edges and arcs in $W$ (with each repeated edge or arc counted as many times as it appears), i.e., $k-1$ is said to be the \textit{length} of $W$. The walk $W$ is said to be \textit{closed} if $v_{n_1}=v_{n_k}$. A mixed graph $G_1$ is said to be a \textit{submixed graph} of $G$ if $V_{G_1}\subseteq V_G$, $E_{G_1}\subseteq E_G$, $\vec{E}_{G_1}\subseteq \vec{E}_G$, $\phi_G|_{E_{G_1}}=\phi_{G_1}$ and $\psi_G|_{\vec{E}_{G_1}}=\psi_{G_1}$. A submixed graph $G_1$ of $G$ is said to be an \textit{induced submixed graph} of $G$ if $G_1$ is the maximal submixed graph of $G$ with vertex set $V_{G_1}$. A submixed graph $G_1$ of $G$ is said to be a \textit{spanning submixed graph} of $G$ if $V_{G_1}=V_G$.

A mixed graph $G$ is said to be a graph if $\vec{E}_G=\emptyset$.
Let $G$ be a graph with $V_G=\{v_1,v_2,\ldots,v_n\}$.
The \textit{adjacency matrix} of $G$, denoted by $A(G)$, is defined as follows: The rows and the columns of $A(G)$ are indexed by the vertices of $G$, and for $i,j=1,2,\ldots,n$,
\begin{eqnarray*}
	\textnormal{the }(v_i,v_j)\textnormal{-th entry of } A(G) =\begin{cases}
		k, & \textnormal{if}~i\neq j\textnormal{ and} ~v_i\overset{k}{\sim} v_j;\\
		2l, & \textnormal{if}~i=j\textnormal{ and} ~v_i\overset{l}{\sim} v_i;\\
		0, & \textnormal{otherwise}.
	\end{cases}
\end{eqnarray*}
We denote its eigenvalues as $\lambda_i(G)$ for $i=1,2,\ldots,n$. Since $A(G)$ is real symmetric, we  can arrange them, with out loss of generality as $\lambda_1(G)\geq\lambda_2(G)\geq\cdots\geq\lambda_{n}(G)$. 
The path graph and the cycle graph on $n$ vertices are denoted by $P_n$ and $C_n$, respectively. The disjoint union of $m$ copies of a graph $G$ is denoted by $mG$.


A mixed graph $G$ is said to be a directed graph if $E_G=\emptyset$. Let $G$ be a directed graph with $V_G=\{v_1,v_2,\ldots,v_n\}$.
The \textit{adjacency matrix} of $G$, denoted by $\vec{A}(G)$, is defined as follows~\cite{bapatbook}: The rows and the columns of $\vec{A}(G)$ are indexed by the vertices of $G$, and for $i,j=1,2,\ldots,n$,
\begin{eqnarray*}
	\textnormal{the }(v_i,v_j)\textnormal{-th entry of } \vec{A}(G) =\begin{cases}
		k, & \textnormal{if}~v_i\overset{k}{\rightarrow} v_j;\\
		0, & \textnormal{otherwise}.
	\end{cases}
\end{eqnarray*}

A \textit{simple graph} (resp. \textit{simple directed graph}) is a graph (resp. directed graph) which has no loops (resp. directed loops) and no multiple edges (resp. multiple arcs). A \textit{simple mixed graph} is a mixed graph which has no loops, directed loops, multiple edges or multiple arcs. A \textit{complete graph} is a simple graph in which every pair of distinct vertices is adjacent. A \textit{complete directed graph} is a simple directed graph in which, for every pair of distinct vertices $u$ and $v$, both $u\rightarrow v$ and $v\rightarrow u$ exist. A \textit{complete mixed graph} is a simple mixed graph in which, for every pair of distinct vertices $u$ and $v$, the following relationship hold: $u\sim v$, $u\rightarrow v$ and $v\rightarrow u$. We denote the complete graph, the complete directed graph and the complete mixed graph on $n$ vertices as $K_n$, $K^D_n$ and $K^M_n$, respectively.

An \textit{independent set} in a mixed graph $G$ is a set of vertices of $G$ in which no two vertices are adjacent. The \textit{independence number} of $G$, denoted by $\alpha(G)$, is the cardinality of a largest independent set of $G$. A \textit{$k$-partite graph} (or \textit{$k$-partite directed graph}, or \textit{$k$-partite mixed graph}, resp.) is a graph (or directed graph, or mixed graph, resp.) whose vertices can be partitioned into $k$ independent sets. If $k=2$, then we say that it is a \textit{bipartite graph} (or \textit{bipartite directed graph}, or \textit{bipartite mixed graph}, resp.). Let $G$ be a  $k$-partite graph with vertex partition $V_G=V_1\overset{.}{\cup}V_2\overset{.}{\cup}\cdots\overset{.}{\cup}V_k$ and $|V_i|=n_i$ for $i=1,2,\ldots,k$. Then $G$ is said to be \textit{complete $k$-partite} if $u\sim v$ for every $u\in V_i$ and $v\in V_j$ with $i\neq j$. We denote it by $K_{n_1,n_2,\ldots,n_k}$. In particular, $K_{k(m)}$ is used when $n_1=n_2=\cdots=n_k=m$.

Let $G$ be a $k$-partite directed graph with vertex partition $V_G=V_1\overset{.}{\cup}V_2\overset{.}{\cup}\ldots\overset{.}{\cup}V_k$ and $|V_i|=n_i$ for $i=1,2,\ldots,k$. Then $G$ is said to be \textit{complete $k$-partite} if $u\rightarrow v$ and $v\rightarrow u$ for every $u\in V_i$ and $v\in V_j$ with $i\neq j$. We denote it by $K^D_{n_1,n_2,\ldots,n_k}$. In particular, $K^D_{k(m)}$ is used when $n_1=n_2=\cdots=n_k=m$.

Let $G$ be a $k$-partite mixed graph with vertex partition $V_G=V_1\overset{.}{\cup}V_2\overset{.}{\cup}\ldots\overset{.}{\cup}V_k$ and $|V_i|=n_i$ for $i=1,2,\ldots,k$. Then $G$ is said to be \textit{complete $k$-partite} if $u\sim v$,  $u\rightarrow v$, and $v\rightarrow u$ for every $u\in V_i$ and $v\in V_j$ with $i\neq j$. We denote it by $K^M_{n_1,n_2,\ldots,n_k}$. In particular, $K^M_{k(m)}$ is used when $n_1=n_2=\cdots=n_k=m$.

Throughout this paper, $I_n$ denotes the $n\times n$ identity matrix, $J_{n\times m}$ denotes the all-ones $n\times m$ matrix, $\mathbf{0}_{n\times m}$ denotes the all-zero $n\times m$ matrix, and $\mathbf{1}_n$ denotes the all-ones matrix of size $n\times 1$. These matrices are simply referred to as $I$, $J$, $\mathbf{0}$, and $\mathbf{1}$, when their sizes are evident from the context. The characteristic polynomial of a square matrix  $M$ is denoted as $P_M(x)$, and the spectrum of $M$ is the multi-set of all the eigenvalues of $M$. For an $n\times n$ real matrix $M$ with real eigenvalues, the eigenvalues are denoted by $\lambda_i(M)$ for $i=1,2,\ldots,n$ and they are ordered as $\lambda_1(M)\geq \lambda_2(M)\geq \cdots\geq\lambda_n(M)$.
An eigenvector $X$ of a matrix is said to be positive if all it's entries are positive real numbers. For a  matrix $M$, $M[i|j]$ denotes the matrix obtained by deleting the $i$-th row and the $j$-th column from $M$. We denote $M[i|i]$ simply as $M[i]$. Let $M=[m_{ij}]$ and $N$ be matrices of order $m\times n$ and $p\times q$, respectively. The Kronecker product of $M$ and $N$, denoted by $M\otimes N$, is the $mp\times nq$ block matrix $[m_{ij}N]$.

In the rest of the paper we consider only mixed graphs having finite number of vertices.
	\section{Integrated adjacency matrix of a mixed graph}\label{S3 adjacency matrix}
Let $G$ be a mixed graph. The \textit{undirected part} (resp. the \textit{directed part}) of $G$ is the spanning submixed graph of $G$ whose edge set is $E_G$, and arc set is $\emptyset$ (resp. edge set is $\emptyset$, and arc set is $\vec{E}_G$). We denote the undirected part   and the directed part of $G$ as $G_u$ and $G_d$, respectively. 
Clearly, every mixed graph can be viewed as the union of its undirected part and directed part. 

\begin{example}\normalfont
	Consider the mixed graph $G$ shown in  Figure~\ref{fig-mixed-exmpl}. It's undirected part $G_u$ and  directed part $G_d$ are also depicted within the same figure.
	\begin{figure}[ht]
		\begin{center}
			\includegraphics[scale=1]{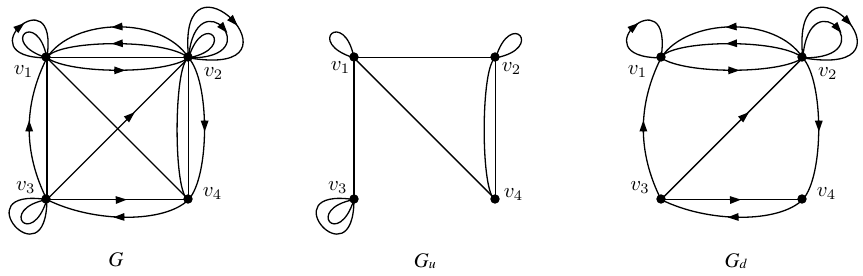}
		\end{center}\caption{A mixed graph $G$ along with its undirected part $G_u$ and directed part $G_d$.}\label{fig-mixed-exmpl}
	\end{figure}
\end{example}
\begin{defn}\normalfont\label{adjacency}
	Let $G$ be a mixed graph with $V_G=\{v_1,v_2,\ldots,v_n\}$. For $i=1,2,\ldots,n$, let $v'_i$ and $v''_i$ represent two copies of $v_i$. We define the \textit{integrated adjacency matrix} of $G$, denoted by $\mathcal{I}(G)$,  as the square matrix of order $2n$, with rows and columns are indexed by the elements of the set $\{v'_1,v'_2,\ldots,v'_n,v''_1,v''_2,\ldots,v''_n\}$.
	For $i,j=1,2,\ldots,n$,
	\begin{equation*}
		\textnormal{the }(v'_i,v'_j)\textnormal{-th entry of }\mathcal{I}(G)=\textnormal{the }(v''_i,v''_j)\textnormal{-th entry of }\mathcal{I}(G)=\begin{cases}
			2r, & \text{if}~i= j\textnormal{ and}~ v_{i}\overset{r}{\sim} v_{j};\\
			s, & \text{if}~i\neq j\textnormal{ and}~ v_{i}\overset{s}{\sim} v_{j};\\
			0, & \text{otherwise};
		\end{cases}
	\end{equation*}
	\begin{equation*}
		\textnormal{the }(v'_i,v''_j)\textnormal{-th entry of }\mathcal{I}(G)=\textnormal{the }(v''_j,v'_i)\textnormal{-th entry of }\mathcal{I}(G)=\begin{cases}
			t, & \text{if}~ v_{i}\overset{t}{\rightarrow} v_{j};\\
			0, & \text{otherwise}.
		\end{cases}
	\end{equation*}
\end{defn}

The eigenvalues of $\mathcal{I}(G)$ are called the $\mathcal{I}$-\textit{eigenvalues} of $G$. We denote these eigenvalues as $\bm{\lambda}_i(G)$ for $i=1,2,\ldots, 2n$. Since $\mathcal{I}(G)$ is real symmetric, these can be arranged, without loss of generality as $\bm{\lambda}_1(G)\geq\bm{\lambda}_2(G)\geq\cdots\geq\bm{\lambda}_{2n}(G)$. 	The spectrum of $\mathcal{I}(G)$ is referred to as the $\mathcal{I}$-\textit{spectrum} of $G$.

It can be seen that a mixed graph $G$ can be  determined by its integrated adjacency matrix $\mathcal{I}(G)$, as long as the indices of the matrix are specified.

By arranging the rows and columns of $\mathcal{I}(G)$ in the order 
$v'_1,v'_2,\ldots,v'_n,v''_1,v''_2,\ldots,v''_n$, the matrix
$\mathcal{I}(G)$ can be viewed as a $2\times 2$ block matrix 
\begin{equation}\label{adjmatrix block}
	\mathcal{I}(G)=\begin{bmatrix}
		A(G_u) & \vec{A}(G_d)\\
		\vec{A}(G_d)^T & A(G_u)
	\end{bmatrix}.
\end{equation}
Since the matrix $\mathcal{I}(G)$ integrates the matrices 	$A(G_u)$ and $\vec{A}(G_d)$, it justifies its name.

Now we define a graph, called the \emph{associated graph} of $G$, denoted by $G^A$, as follows.
Let $V_{G^A}=\{v'_1,v'_2,\ldots,v'_n,v''_1,v''_2,\ldots,v''_n\}$. The adjacency among the vertices of $G^A$ is determined using the following rules, for $i,j=1,2,\ldots,n$: 
\begin{itemize}
	\item[(i)] $v'_i\overset{k}{\sim} v'_j$ and $v''_i\overset{k}{\sim} v''_j$ in $G^A$ if and only if $v_i\overset{k}{\sim} v_j$ in $G$;
	\item[(ii)] $v'_i\overset{k}{\sim} v''_j$ in $G^A$ if and only if $v_i\overset{k}{\rightarrow} v_j$ in $G$.
\end{itemize}

Clearly, $G^A$ has $2e(G)+a(G)$ edges (excluding loops) and $2l(G)$  loops. The subgraphs spanned by the vertex subsets $\{v'_1,v'_2,\ldots,v'_n\}$ and $\{v''_1,v''_2,\ldots,v''_n\}$ of $G^A$ are isomorphic to $G_u$. If $G$ has a loop at $v_i$,  then correspondingly $G^A$ has two loops: one at $v'_i$ and the other at $v''_i$. If $G$ has a directed loop at $v_i$ in $G$, then correspondingly $G^A$ has an edge joining $v'_i$ and $v''_i$. 

It is evident that $A(G^A)=\mathcal{I}(G)$. Consequently the spectrum of $\mathcal{I}(G)$ is identical to the spectrum of $A(G^A)$. From the construction of $G^A$ it is easy to observe that if $G$ is simple, then $G^A$ is a simple graph.
\begin{example}\normalfont
	Consider the graph $G$ shown in Figure~\ref{fig-mixed-exmpl}. Its integrated adjacency matrix is 
	$$\mathcal{I}(G)=\begin{bmatrix}
		2&1&1&1&1&1&0&0\\
		1&2&0&2&2&2&0&1\\
		1&0&4&0&1&1&0&1\\
		1&2&0&0&0&0&1&0\\
		1&2&1&0&2&1&1&1\\
		1&2&1&0&1&2&0&2\\
		0&0&0&1&1&0&4&0\\
		0&1&1&0&1&2&0&0
	\end{bmatrix},$$ and its associated graph $G^A$ is shown in Figure~\ref{associated graph}.
	\begin{figure}[ht]
		\begin{center}
			\includegraphics[scale=1]{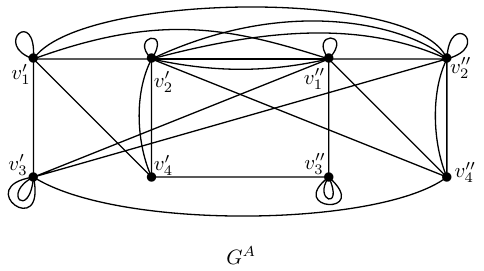}
		\end{center}\caption{The associated graph $G^A$ of the mixed graph $G$ shown in Figure~\ref{fig-mixed-exmpl}}\label{associated graph}
	\end{figure}
\end{example}
	
\section{Some basic results}\label{S4basic results}
We start with the following observations.
\begin{observation}\label{observations}\normalfont
	\begin{enumerate}[(1)]
		\item If $G$ is a graph, then $\mathcal{I}(G)=I_2\otimes A(G)$. Consequently, the eigenvalues of $\mathcal{I}(G)$ are identical to those of $A(G)$ but with twice their multiplicities. That is, $\lambda$ is an eigenvalue of $A(G)$ with multiplicity $k$ if and only if $\lambda$ is an eigenvalue of $\mathcal{I}(G)$ with multiplicity $2k$.
		\item If $G$ is a directed graph in which every pair of its vertices has the same number of arcs in both direction, then $\mathcal{I}(G)=(J_2-I_2)\otimes\vec{A}(G)$.
		As a result, $\lambda_1,\lambda_{2},\ldots,\lambda_n$ are the eigenvalues of $\vec{A}(G)$ if and only if  $\lambda_1,\lambda_{2},\ldots,\lambda_n,-\lambda_1,-\lambda_{2},\ldots,-\lambda_n$ are the eigenvalues of $\mathcal{I}(G)$, where $n$ represents the number of vertices in $G$.
	\end{enumerate}
\end{observation}
The proof of the following result  follows directly from the definition of the integrated adjacency matrix of a mixed graph.
\begin{pro}\label{pro-degreesum}
	Let $G$ be a mixed graph. Then the following holds.
	\begin{enumerate}[(i)]
		\item $\underset{u\in V_G}{\sum}d(u)=2(e(G)+l(G))$,
		\item $\underset{u\in V_G}{\sum}d^+(u)=a(G)=\underset{u\in V_G}{\sum}d^-(u)$.
	\end{enumerate}
\end{pro}
\begin{thm}\label{thm-traceA}
	Let $G$ be a simple mixed graph on $n$ vertices. Then the following holds.
	\begin{enumerate}[(i)]
		\item $\underset{i=1}{\overset{2n}{\sum}}\bm{\lambda}_i(G)=0$,
		\item $\underset{i=1}{\overset{2n}{\sum}}\bm{\lambda}_i(G)^2=4e(G)+2a(G)$.
	\end{enumerate} 
\end{thm}
\begin{proof}
	Since $\underset{i=1}{\overset{2n}{\sum}}\bm{\lambda}_i(G)=tr(\mathcal{I}(G))$, part~(i) follows. 
	
	Let $v_1,v_2,\ldots,v_n$ be the vertices of $G$. Note that the $i$-th diagonal entry of $\mathcal{I}(G)^2$ is $d(v_i)+d^+(v_i)$, if $i=1,2,\ldots,n$, and $d(v_{i-n})+d^-(v_{i-n})$, if $i=n+1,n+2,\ldots,2n$. 
	Since $\underset{i=1}{\overset{2n}{\sum}}\bm{\lambda}_i(G)^2=tr(\mathcal{I}(G)^2)$, we have
	\begin{eqnarray}\label{sum}
		\underset{i=1}{\overset{2n}{\sum}}\bm{\lambda}_i(G)^2
		&=&\underset{i=1}{\overset{n}{\sum}}(d(v_i)+d^+(v_i))+\underset{i=n+1}{\overset{2n}{\sum}}(d(v_{i-n})+d^-(v_{i-n}))\nonumber\\
		&=&\underset{i=1}{\overset{n}{\sum}}(2d(v_i)+d^+(v_i)+d^-(v_i)).
	\end{eqnarray}
	Part~(ii) follows by applying Proposition~\ref{pro-degreesum} in \eqref{sum}.
\end{proof}
\begin{thm}\label{thm-interlace}
	Let $G$ be a simple mixed graph, and let $H$ be an induced submixed graph of $G$. Then the $\mathcal{I}$-eigenvalues of $H$ interlace the $\mathcal{I}$-eigenvalues of $G$.
\end{thm}
\begin{proof}
	Since $H^A$ is an induced subgraph of $G^A$, as per the interlacing theorem for simple graphs, it follows that the eigenvalues of $A(H^A)$ interlace the eigenvalues of $A(G^A)$. Since $\mathcal{I}(G)=A(G^A)$ and $\mathcal{I}(H)=A(H^A)$, the result follows.
\end{proof}

\begin{thm}\label{coeff}
	For a simple mixed graph $G$ on $n$ vertices, the coefficient of $x^{2n-2}$ in the characteristic polynomial of $\mathcal{I}(G)$ is $-2e(G)-a(G)$.
\end{thm}
\begin{proof}
	The coefficient of $x^{2n-2}$ in the characteristic polynomial of $\mathcal{I}(G)$ is $\underset{1\leq i<j\leq2n}{\sum}\bm{\lambda}_i(G)\bm{\lambda}_j(G)$. Notice that $$\left(\underset{j=1}{\overset{2n}{\sum}}\bm{\lambda}_j(G)\right)^2=\underset{j=1}{\overset{2n}{\sum}}\bm{\lambda}_j(G)^2+2\underset{1\leq j<k\leq2n}{\sum}\bm{\lambda}_j(G)\bm{\lambda}_k(G).$$
	By using Theorem~\ref{thm-traceA} in the above equation, we get the result as desired.
\end{proof}

\begin{defn}\normalfont
	Let $G$ be a mixed graph with at least one arc. We say a walk containing an arc in $G$ as an \textit{alternating walk} if starting from the first arc it's subsequent arcs must alternate between forward and reverse directions.
	
	We say an alternating walk $W$ in $G$ as an \textit{alternating path} if it satisfies the following conditions:
	\begin{enumerate}[(i)]
		\item Each vertex in $W$ occurs at most twice.
		\item If a vertex in $W$ occurs twice, then the number of arcs between its two occurrences must be odd.
		\item Each edge in $W$ occurs at most twice.
		\item If an edge in $W$ occurs twice, then the number of arcs between its two occurrences must be odd. 
		\item Each arc in $W$ occurs exactly once.
	\end{enumerate}
	
	We say an alternating walk $W$ in $G$ as an \textit{alternating cycle} if it satisfies the following conditions:
	\begin{enumerate}[(1)]
		\item The starting and ending vertex of $W$ are the same.
		\item $W$ satisfies the above conditions~(iii)-(v).
		\item The vertices in $W$ other than the starting vertex (ending vertex) satisfy the above conditions~(i) and (ii).
		\item The starting vertex in $W$ occurs at most thrice.
		\item If the starting vertex in $W$ occurs thrice, then the number of arcs between its successive occurrences must be odd.
	\end{enumerate}
\end{defn}

\begin{thm}\label{entry}
	Let $G$ be a simple mixed graph on $n$ vertices $v_1,v_2,\ldots,v_n$. Then we have the following for $i,j=1,2,\ldots,n$.
	\begin{enumerate}[(i)]
		\item The $(v'_i,v'_j)$-th entry of $\mathcal{I}(G)^k$ is the sum of the number of walks in $G$ from $v_i$ to $v_j$ of length $k$ containing no arcs, and the number of alternating walks in $G$ from $v_i$ to $v_j$ of length $k$ containing even number of arcs with the first arc being directed in the forward direction;
		\item The $(v'_i,v''_j)$-th entry of $\mathcal{I}(G)^k$ is the number of alternating walks in $G$ from $v_i$ to $v_j$ of length $k$ containing odd number of arcs with the first arc being directed in the forward direction;
		\item The $(v''_i,v'_j)$-th entry of $\mathcal{I}(G)^k$ is the number of alternating walks in $G$ from $v_i$ to $v_j$ of length $k$ containing odd number of arcs with the first arc being directed in the backward direction;
		\item The $(v''_i,v''_j)$-th entry of $\mathcal{I}(G)^k$ is the sum of the number of walks in $G$ from $v_i$ to $v_j$ of length $k$ containing no arcs, and the number of alternating walks in $G$ from $v_i$ to $v_j$ of length $k$ containing even number of arcs with the first arc being directed in the backward direction.
	\end{enumerate}
\end{thm}
\begin{proof}
	$G^A$ is a simple graph on $2n$ vertices $v'_1,v'_2\ldots,v'_n,v''_1,v''_2,\ldots,v''_n$. According to  \cite[Proposition~1.3.4]{cvetkovic}: ``Let $G$ be a simple graph with $V_G=\{v_1,v_2,\ldots,v_n\}$. Then for $i,j=1,2,\ldots,n$ the $(v_i,v_j)$-th entry of $A(G)^k$ is equal to the number of walks in $G$ of length $k$ that start at $v_i$ and end at $v_j$'', we have the following for $i,j=1,2,\ldots,n$. 
	\begin{enumerate}[(a)]
		\item The $(v'_i,v'_j)$-th entry of $A(G^A)^k$ is the number of walks in $G^A$ of length $k$ that start at $v'_i$ and end at $v'_j$;
		\item The $(v'_i,v''_j)$-th entry of $A(G^A)^k$ is the number of walks in $G^A$ of length $k$ that start at $v'_i$ and end at $v''_j$;
		\item The $(v''_i,v'_j)$-th entry of $A(G^A)^k$ is the number of walks in $G^A$ of length $k$ that start at $v''_i$ and end at $v'_j$;
		\item The $(v''_i,v''_j)$-th entry of $A(G^A)^k$ is the number of walks in $G^A$ of length $k$ that start at $v''_i$ and end at $v''_j$.
	\end{enumerate}
	
	We show that the numbers given in (a), (b), (c) and (d) are equal to the numbers given in parts (i), (ii), (iii) and (iv), respectively in the theorem.
	
	First we show that the number given in (a) is equal to the number given in part~(i) of the theorem. 
	Consider the vertex partition $V_{G^A}=A\cup B$, where $A=\{v_1',v_2',\ldots,v_n'\}$ and $B=\{v_1'',v_2'',\ldots,v_n''\}$. Fix $i,j\in\{1,2,\ldots,n\}$. 	Let $W$ be a walk in $G$ from $v_i$ to $v_j$ of length $k$ as in part (i) of the theorem. We prove that there is a walk in $G^A$ corresponding to $W$ of length $k$ which starts at $v'_i$ and ends at $v'_j$.
	
	\textit{Case~1.} Suppose that $W$ has no arcs.
	As per the construction of $G^A$, $W$ corresponds to two distinct walks of length $k$ in $G^A$. One walk joins only the vertices of $A$, and the other joins only the vertices of $B$. For the count, we consider the walk, say $W'$ which joins the vertices of $A$ only. Then $W'$ starts at $v'_i$ and ends at $v'_j$.
	
	\textit{Case~2.} Suppose that $W$ is an alternating walk containing even number of arcs with the first arc being in the forward direction. Based on the construction of $G^A$, $W$ corresponds to a walk, say $W'$ in $G^A$ of length $k$, which is found as follows. 
	Since the first arc of $W$ is in the forward direction, the corresponding edge in $G^A$ joins a vertex in $A$ to a vertex in $B$. Therefore, the vertices and the edges prior to the first arc in $W$ are associated to the corresponding vertices and the edges in $A$, implying that $W'$ starts at $v_i'$. As the arcs of $W$ alternate in direction, the second arc in $W$ must be in the backward direction consequently the corresponding edge in $G^A$ joins a vertex in $B$ to a vertex in $A$. The vertices and the edges between first two arcs in $W$ are associated to the corresponding vertices and the edges in $B$. Continuing like this, we have the edge corresponding to the last arc of $W$ which joins a vertex in $B$ to a vertex in $A$, since there are even number of arcs in $W$. If there are any remaining edges in $W$ after the last arc, then the corresponding vertices and edges are in $A$. This implies that $W'$ starts at $v_i'$ and ends at $v_j'$.
	
	Conversely, let $W'$ be a walk in $G^A$ of length $k$ which starts at $v_i'$ and ends at $v_j'$. We prove that there is a walk or an alternating walk in $G$ as in part~(i) of the theorem.
	
	\textit{Case~I.} Suppose all the edges of $W'$ join the vertices of $A$ only. Then from the construction of $G^A$, it corresponds to a walk of length $k$ joining $v_i$ and $v_j$ in $G$ containing no arc. 
	
	\textit{Case~II.} Suppose an edge of $W'$ joins a vertex in $A$ and a vertex in $B$. As per the construction of $G^A$, $W'$ corresponds to a walk, say $W$ in $G$ of length $k$, which is found as follows. Since both $v_i'$ and $v_j'$ are in $A$, there must be even number of edges in $W'$ which join vertices in $A$ with vertices in $B$. The first such edge in $W'$ joins a vertex in $A$ with a vertex in $B$. This edge corresponds to an arc in $G$ with forward direction, because $v_i'\in A$. The subsequent edge in $W'$ which joins a vertex in $B$ with a vertex in $A$ corresponds to an arc in $G$ with a backward direction. This alternation continues with the even number of edges in $W'$ joining the vertices in $A$ with the vertices in $B$ are associated to the arcs in $G$ with alternating directions. The remaining edges in $W'$ which are not part of the alternating sequence, correspond to the edges in $G$. Thus, $W$ is an alternating walk of length $k$ in $G$ from $v_i$ to $v_j$ containing even number of arcs, with the first arc in the forward direction.
	
	Hence the sum of the number of walks in $G$ from $v_i$ to $v_j$ of length $k$ containing no arcs, and the number of alternating walks in $G$ from $v_i$ to $v_j$ of length $k$ containing only even number of arcs with the first arc in the forward direction, is equal to the number of walks in $G^A$ of length $k$ that starts at $v'_i$ and ends at $v'_j$. 
	
	In similar manner, we can prove that the numbers given in (b), (c) and (d) are equal to the numbers given in parts (ii), (iii) and (iv), respectively in the theorem.
	
	Since $\mathcal{I}(G)=A(G^A)$, we have $\mathcal{I}(G)^k=A(G^A)^k$. Hence the result follows.
\end{proof}
With the help of the proof of Theorem~\ref{entry}, we can observe the following.
\begin{note}\normalfont
	\begin{enumerate}[(a)]
		\item Each alternating path $W$ in $G$ corresponds to a path $W'$ in $G^A$, since by definition each repeated vertex and each repeated edge in $W$ corresponds to its two distinct copies in $W'$.
		
		\item Each path $W'$ in $G^A$ corresponds to either a path or an alternating path in $G$, depending on whether $W'$ contains no edge joining the vertices of $A$ and $B$ (path) or at least one such edge (alternating path).
		
		\item Each alternating cycle $W$ with an odd number of arcs in $G$ corresponds to a path $W'$ in $G^A$, since the odd number of arcs ensures that the starting vertex does not occur in the interior of the cycle. Thus $W$ is an alternating path, and by (a), $W'$ is a path.
		
		\item Each alternating cycle $W$ with an even number of arcs in $G$ corresponds to a cycle $W'$ in $G^A$, since by definition each repeated vertex (excluding the terminal vertex) and each repeated edge in $W$ corresponds to its two distinct copies in $W'$. The even number of arcs guarantees that $W'$ is a cycle.
		
		\item Each cycle $W'$ in $G^A$ corresponds to either a cycle or an alternating cycle with an even number of arcs in $G$, depending on whether $W'$ contains no edge joining the vertices of $A$ and $B$ (cycle) or at least one such edge (alternating cycle).
	\end{enumerate}
\end{note}
Let $G$ be a mixed graph. An \textit{odd cycle} in $G$ is a cycle having odd number of edges. An \textit{odd alternating cycle}  in $G$ is an alternating cycle having odd number of arcs and edges.

\begin{defn}\normalfont
	A loopless mixed graph $G$ is said to satisfy the \textit{associated-bipartite property} (shortly, \textit{AB property}), if it has no odd cycle  and no odd alternating cycle with even number of arcs.
\end{defn}
\begin{lemma}\label{ABbipartite}
	A loopless mixed graph $G$ has the AB property if and only if $G^A$ is bipartite.
\end{lemma}
\begin{proof}
	Let $G$ be a mixed graph with $V_G=\{v_1,v_2,\ldots,v_n\}$. Then $V_{G^A}=\{v'_1,v'_2,\ldots,v'_n,$ $v''_1,v''_2,\ldots,v''_n\}$. Let $A_1=\{v'_1,v'_2,\ldots,v'_n\}$ and $A_2=\{v''_1,v''_2,\ldots,v''_n\}$. Suppose $G$ has the AB property. Then $G$ has no odd cycle and no odd alternating cycle with even number of arcs. 
	
	An odd cycle (if exist) say $C$ in $G^A$ satisfies one of the following.
	
	(i) $C$ contains only the vertices of $A_1$ or only the vertices of $A_2$;
	
	(ii) $C$ contains at least one vertex from each of the sets $A_1$ and $A_2$, and even number of edges which join the vertices in $A_1$ with the vertices in $A_2$.
	
	Since $G$ has no odd cycle, $G^A$ has no odd cycle as in (i). Since $G$ has no odd alternating cycle with even number of arcs, $G^A$ has no odd cycle as in (ii). As per the above, we have that $G^A$ has no odd cycle, i.e., $G^A$ is bipartite. By retracing the above, the converse can be proved.
\end{proof}
\begin{thm}
	Let $G$ be a simple mixed graph on $n$ vertices. Then the following conditions are equivalent.
	\begin{enumerate}[(i)]
		\item $G$ has the AB property;
		\item if $P_{\mathcal{I}(G)}(x)=x^{2n}+c_1x^{2n-1}+c_2x^{2n-2}+\cdots+c_{2n}$ is the characteristic polynomial of $\mathcal{I}(G)$, then $c_{2k+1}=0$ for $k=0,1,\ldots,n-1$;
		\item the $\mathcal{I}$-eigenvalues of $G$ are symmetric with respect to the origin, i.e., if $\lambda$ is an eigenvalue of $\mathcal{I}(G)$ with multiplicity $k$, then $-\lambda$ is also an eigenvalue of $\mathcal{I}(G)$ with multiplicity $k$.
	\end{enumerate} 
\end{thm}
\begin{proof}
	Since the spectrum of $\mathcal{I}(G)$ coincides with the spectrum of $A(G^A)$, the result directly follows from Lemma~\ref{ABbipartite} and \cite[Theorem~3.14]{bapatbook}.
\end{proof}
\section{$\mathcal{I}$-spectra of certain special mixed graphs}\label{S5 spectra}
\begin{thm}\label{thm-rs regular}
	Let $G$ be a mixed graph on $n$ vertices. Then $G$ is an $(r,s)$-regular mixed graph if and only if $r+s$ and $r-s$ are  eigenvalues of $\mathcal{I}(G)$ with corresponding eigenvectors $\boldsymbol{1}_{2n}$ and $\begin{bmatrix}
		\boldsymbol{1}_n^T&
		-\boldsymbol{1}_n^T
	\end{bmatrix}^T$, respectively.
\end{thm}
\begin{proof}
	Since $G$ is a mixed graph, we can express $\mathcal{I}(G)$ as a $2\times 2$ block matrix as described in \eqref{adjmatrix block}. 
	Suppose $G$ is $(r,s)$-regular, the row sums of the blocks $A(G_u)$, $\vec{A}(G_d)$, and $\vec{A}(G_d)^T$ are $r,s$, and $s$, respectively. Consequently $r+s$ and $r-s$ are eigenvalues of $\mathcal{I}(G)$ with corresponding eigenvectors $\boldsymbol{1}_{2n}$ and $\begin{bmatrix}
		\boldsymbol{1}_n^T&
		-\boldsymbol{1}_n^T
	\end{bmatrix}^T$, respectively. 
	
	Conversely, assume that $r+s$ and $r-s$ are  eigenvalues of $\mathcal{I}(G)$ with corresponding eigenvectors $\boldsymbol{1}_{2n}$ and $\begin{bmatrix}
		\boldsymbol{1}_{n}^T&
		-\boldsymbol{1}_n^T
	\end{bmatrix}^T$, respectively. Then for all $u\in V_G$, the following conditions hold: $d(u)+d^+(u)=r+s$, $d(u)+d^-(u)=r+s$, $d(u)-d^+(u)=r-s$, and $d^-(u)-d(u)=-(r-s)$. These equations imply that $d(u)=r$ and $d^+(u)=d^-(u)=s$ for all $u\in V_G$. Hence $G$ is $(r,s)$-regular.
\end{proof}

Let $G$ be a mixed graph. For $v\in V_G$, $G-v$ denotes the mixed graph obtained by deleting the vertex $v$ along with all edges and arcs incident with $v$ in $G$.
For $e\in E_G$, $G-e$ represents the mixed graph obtained by deleting the edge $e$ from $G$.
Similarly, for $a\in \vec{E}_G$, $G-a$ denotes the mixed graph obtained by deleting the arc $a$ from $G$.

For $u,v\in V_G$, we denote
\begin{itemize}
	\item $u\leftrightarrows v$, if there is exactly one arc from $u$ to $v$, and exactly one arc from $v$ to $u$;
	\item $u\ccombinedarroww v$, if there is exactly one arc from $u$ to $v$, and exactly one edge joining $u$ and $v$;
	\
	\item  $u\ccombinedarrowss v$, if there is exactly one arc from $u$ to $v$, and exactly one arc from $v$ to $u$, and exactly one edge joining $u$ and $v$.
\end{itemize}
\begin{thm}
	Let $G$ be a mixed graph with $V_G=\{v_1,v_2,\ldots,v_n\}$. Let $G_j$ be the mixed graph obtained from $G$ by taking a new vertex $u$, and joining it with $v_j$ in one of the following way: $v_j\overset{1}{\rightarrow} u$, $u\overset{1}{\rightarrow} v_j$, $v_j\overset{1}{\sim} u$, $v_j\rightleftarrows u$, $v_j\ccombinedarroww u$, $u\ccombinedarroww v_j$ and $v_j\ccombinedarrowss u$. Then we have the following.
	\begin{enumerate}[(i)]
		\item If $v_j\overset{1}{\rightarrow} u$, then $P_{\mathcal{I}(G_j)}(x)=x^2P_{\mathcal{I}(G)}(x)-xP_{\mathcal{I}(G)[j]}(x).$
		\item If $u\overset{1}{\rightarrow} v_j$, then $P_{\mathcal{I}(G_j)}(x)=x^2P_{\mathcal{I}(G)}(x)-xP_{\mathcal{I}(G)[n+j]}(x).$
		\item If $v_j\overset{1}{\sim} u$, then $P_{\mathcal{I}(G_j)}(x)=x^2P_{\mathcal{I}(G)}(x)-x\{P_{\mathcal{I}(G)[j]}(x)+P_{\mathcal{I}(G)[n+j]}(x)\}+P_{\mathcal{I}(G-v_j)}(x).$
		\item If $v_j\rightleftarrows u$, then $P_{\mathcal{I}(G_j)}(x)=x^2P_{\mathcal{I}(G)}(x)-x\{P_{\mathcal{I}(G)[j]}(x)+P_{\mathcal{I}(G)[n+j]}(x)\}+P_{\mathcal{I}(G-v_j)}(x).$
		\item If $v_j\ccombinedarroww u$, then $P_{\mathcal{I}(G_j)}(x)=x^2P_{\mathcal{I}(G)}(x)-x\{2P_{\mathcal{I}(G)[j]}(x)+P_{\mathcal{I}(G)[n+j]}(x)+(-1)^nP_{\mathcal{I}(G)[n+j|j]}(x)+(-1)^nP_{\mathcal{I}(G)[j|n+j]}(x)\}+P_{\mathcal{I}(G-v_j)}(x).$
		\item If $u\ccombinedarroww v_j$, then $P_{\mathcal{I}(G_j)}(x)=x^2P_{\mathcal{I}(G)}(x)-x\{P_{\mathcal{I}(G)[j]}(x)+2P_{\mathcal{I}(G)[n+j]}(x)+(-1)^nP_{\mathcal{I}(G)[n+j|j]}(x)+(-1)^nP_{\mathcal{I}(G)[j|n+j]}(x)\}+P_{\mathcal{I}(G-v_j)}(x).$
		\item If $v_j\ccombinedarrowss u$, then $P_{\mathcal{I}(G_j)}(x)=x^2P_{\mathcal{I}(G)}(x)-2x\{P_{\mathcal{I}(G)[j]}(x)+P_{\mathcal{I}(G)[n+j]}(x)+(-1)^nP_{\mathcal{I}(G)[n+j|j]}(x)+(-1)^nP_{\mathcal{I}(G)[j|n+j]}(x)\}.$
	\end{enumerate}
\end{thm}
\begin{proof}
	If $v_j\overset{1}{\rightarrow} u$, then by a suitable ordering of vertices of $G_j$, we have
	$$\mathcal{I}(G_j)=\begin{bmatrix}
		A(G_u) & \mathbf{0}_{n\times 1} & \vec{A}(G_d) &E_j\\
		\mathbf{0}_{1\times n} & 0 &\mathbf{0}_{1\times n} & 0\\
		\vec{A}(G_d)^T & \mathbf{0}_{n\times 1} &  A(G_u) & \mathbf{0}_{n\times 1}\\
		E_j^T & 0 & \mathbf{0}_{1\times n} & 0
	\end{bmatrix},
	$$
	where $E_j$ denotes the $n\times 1$ matrix with $(j,1)$-th entry is $1$ and all other entries are $0$. Now
	\begin{eqnarray}
		P_{\mathcal{I}(G_j)}(x)&=&|xI-\mathcal{I}(G_j)| \nonumber\\
		&=&\left|\begin{matrix}
			xI-A(G_u) & \mathbf{0}_{n\times 1} & -\vec{A}(G_d) & -E_j\\
			\mathbf{0}_{1\times n} & x & \mathbf{0}_{1\times n} & 0\\
			-\vec{A}(G_d)^T & \mathbf{0}_{n\times 1} &  xI-A(G_u) & \mathbf{0}_{n\times 1}\\
			-E_j^T & 0 & \mathbf{0}_{1\times n} & x
		\end{matrix}\right| \nonumber\\
		&=&(-1)^{(n+1)+(n+1)}x\left|\begin{matrix}
			xI-A(G_u) & -\vec{A}(G_d) & -E_j\\
			-\vec{A}(G_d)^T &  xI-A(G_u) & \mathbf{0}_{n\times 1}\\
			-E_j^T & \mathbf{0}_{1\times n} & x
		\end{matrix}\right| \nonumber\\
		&=&x\left[(-1)^{(2n+1)+(2n+1)}x\left|\begin{matrix}
			xI-A(G_u) & -\vec{A}(G_d)\\
			-\vec{A}(G_d)^T & xI-A(G_u)
		\end{matrix}\right|\right.\nonumber\\
		& & \left.+(-1)^{(2n+1)+j}\times(-1)\left|\begin{matrix}
			(xI-A(G_u))[0|j] & -\vec{A}(G_d) & -E_j\\
			-\vec{A}(G_d)^T[0|j] & xI-A(G_u) & \mathbf{0}_{n\times 1}
		\end{matrix}\right|\right]\nonumber\\
		&=& x^2P_{\mathcal{I}(G)}(x)-xP_{\mathcal{I}(G)[j]}(x). \nonumber
	\end{eqnarray}
	
	This completes the proof of part~(i). Similar to part~(i), the remaining parts can also be proved.
\end{proof}


\begin{thm}\label{thm-adj-complete-kpartite}
	\begin{enumerate}[(i)]
		\item The $\mathcal{I}$-spectrum of $K^M_{k(m)}$ is $2m(k-1)$ with multiplicity $1$, $-2m$ with multiplicity $k-1$, and $0$ with multiplicity $(2m-1)k$.
		\item The $\mathcal{I}$-spectrum of $K^D_{k(m)}$ is $\pm m(k-1)$ with multiplicity $1$, $\pm m$ with multiplicity $k-1$, and $0$ with multiplicity $2k(m-1)$.
	\end{enumerate}
\end{thm}
\begin{proof}
	We have $\mathcal{I}(K^M_{k(m)})=J_2\otimes[(J_k-I_k)\otimes J_m]$ and $\mathcal{I}(K^D_{k(m)})=(J_2-I_2)\otimes[(J_k-I_k)\otimes J_m]$. For a positive integer $s$, the eigenvalues of $J_s$ are $s$ with multiplicity $1$, and $0$ with multiplicity $s-1$. The proof follows by applying these in
	\cite[Lemma~3.24]{bapatbook}: ``Let $A$ and $B$ be symmetric matrices of order $m\times m$ and $n\times n$, respectively. Then the eigenvalues of $A\otimes B$ are given by $\lambda_i(A)\lambda_j(B)$ for $i=1,2,\ldots,m$ and $j=1,2,\ldots,n$''.
\end{proof}

\begin{cor}
	\begin{itemize}
		\item[(i)] The $\mathcal{I}$-spectrum of $K^M_n$ is $2(n-1)$ with multiplicity $1$, $-2$ with multiplicity $n-1$, and $0$ with multiplicity $n$.
		\item[(ii)] The $\mathcal{I}$-spectrum of $K^D_n$ is $\pm(n-1)$ with multiplicity $1$, and $\pm 1$ with multiplicity $n-1$.
	\end{itemize}
\end{cor}
\begin{proof}
	By taking $m=1$, and $k=n$ in Theorem~\ref{thm-adj-complete-kpartite}, the result follows. 
\end{proof}
An \textit{oriented graph} of a graph $G$ is the directed graph obtained from $G$ by replacing each edge in $G$ by an arc. 
\begin{thm}\label{cycle spectrum}
	\begin{enumerate}[(i)]
		\item If $G$ is the oriented graph of $P_n$ $(n\geq 2)$ with all its arcs have the same direction, then the $\mathcal{I}$-spectrum of $G$ is $1$, $-1$, and $0$ with multiplicity $n-1$, $n-1$, and $2$, respectively.
		\item If $G$ is the oriented graph of $P_n$ $(n\geq 2)$ with all its arcs have the alternating direction, then the $\mathcal{I}$-spectrum of $G$ is $0$ with multiplicity $n$, and $2\cos\frac{\pi k}{n+1}$ for $k=1,2,\ldots,n$.
		\item If $G$ is the oriented graph of $C_n$ $(n\geq 3)$ with all its arcs have the same direction, then the $\mathcal{I}$-spectrum of $G$ is $1$ and $-1$ each with multiplicity $n$.
		\item If $G$ is the oriented graph of $C_n$ $(n\geq 4)$ with all its arcs have the alternating direction, then the $\mathcal{I}$-spectrum of $G$ is $0$ with multiplicity $n$, and $2\cos\frac{2\pi k}{n}$ for $k=1,2,\ldots,n$.
	\end{enumerate}
\end{thm}
\begin{proof}
	\begin{enumerate}[(i)]
		\item Here $G^A$ is the union of $(n-1)K_2$, and two isolated vertices. Consequently the spectrum of $A(G^A)$ is $1$, $-1$, and $0$ with multiplicity $n-1$, $n-1$, and $2$, respectively.
		
		\item Here $G^A$ is the union of $P_n$, and $n$ isolated vertices. It is known that the spectrum of $A(P_n)$ is $2\cos\frac{\pi k}{n+1}$ for $k=1,2,\ldots,n$. So, the spectrum of $A(G^A)$ is $0$ with multiplicity $n$, and $2\cos\frac{\pi k}{n+1}$ for $k=1,2,\ldots,n$.
		
		\item Here $G^A=nK_2$ and so the spectrum of $A(G^A)$ is $1$ and $-1$ each with multiplicity $n$.
		
		\item Here $G^A$ is the union of $C_n$ and $n$ isolated vertices. It is known that the spectrum of $A(C_n)$ is $2\cos\frac{2\pi k}{n}$ for $k=1,2,\ldots,n$.	 So, the spectrum of $A(G^A)$ is $0$ with multiplicity $n$, and $2\cos\frac{2\pi k}{n}$ for $k=1,2,\ldots,n$. 
	\end{enumerate}
	Since the spectrum of $\mathcal{I}(G)$ and $A(G^A)$ are the same, we get the result.
	%
\end{proof}
\section{Mixed components of a mixed graph}\label{S6 components}
\begin{defn}\normalfont
	Let $G$ be a mixed graph. A submixed graph $G'$ of $G$ having at least one arc is said to be a \textit{special submixed graph} of $G$, if it satisfies the following conditions.
	\begin{enumerate}[(i)]
		\item Any two distinct vertices of $G'$ are joined by an alternating walk;
		\item If a component $H$ of $G'_u$ has at least one vertex having in-arc in $G'$, and at least one vertex having out-arc in $G'$, then for each vertex $v$ of $H$, there exists a closed alternating walk in $G'$ which starts and ends at $v$, and containing odd number of arcs.
	\end{enumerate}
\end{defn}

\begin{defn}\label{ASM}\normalfont
	Let $G$ be a mixed graph. A submixed graph $G'$ of $G$ is said to be a \textit{mixed component} of $G$ if it satisfies one of the following three conditions.
	\begin{enumerate}[(i)]
		\item $G'$ is a component of $G_u$, and each vertex of $G'$ has out-degree zero in $G$.
		\item $G'$ is a component of $G_u$, and each vertex of $G'$ has in-degree zero in $G$.
		\item $G'$ is a maximal special submixed graph of $G$.
	\end{enumerate}
\end{defn}

We call a mixed component satisfying the conditions (i), (ii) and (iii) mentioned above as a \textit{mixed component of Type~I, Type~II} and \textit{Type~III}, respectively. 

\begin{example}\normalfont\label{exampmixed component}
	Consider the mixed graph $G$ shown in Figure~\ref{ASMfig}. The mixed graphs $G_1,G_2,G_3,$ and $G_4$ shown in the same figure are some special submixed graphs of $G$. The mixed graphs $G_2$, $G_4$, $G_5$ and $G_6$ are the mixed components of $G$. Since $G_1$ and $G_3$ are submixed graphs of $G_2$, they are not maximal special submixed graphs of $G$. So, they are not mixed components of $G$. Furthermore, $G_5$ and $G_6$ are not special submixed graphs of $G$, since they have no arcs.
	\begin{figure}[ht]
		\begin{center}
			\includegraphics[scale=1]{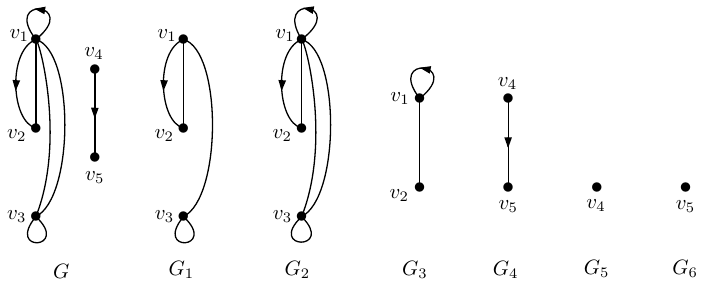}
		\end{center}\caption{A mixed graph $G$ and it's special submixed graphs $G_1$, $G_2$, $G_3$, $G_4$, $G_5$ and $G_6$}\label{ASMfig}
	\end{figure}
\end{example}	

\noindent\textbf{Particular cases}:
The structure of the mixed components of a mixed graph $G$ is detailed below for the cases when $G$ is either a directed graph or a graph.
\begin{enumerate}[(1)]
	\item If $G$ is a directed graph:
	\begin{enumerate}[(i)]
		\item A mixed component of Type~I in $G$ is a vertex $v\in V_G$ whose out-degree in $G$ is zero.
		\item A mixed component of Type~II in $G$ is a vertex $v\in V_G$ whose in-degree in $G$ is zero.
		\item A mixed component of Type~III in $G$ is a non-trivial subdirected graph $G'$ of $G$ which is maximal with respect to the following conditions:
		\begin{enumerate}[(a)]
			\item Any two distinct vertices in $G'$ are connected by an alternating walk in $G'$.
			\item For each vertex $v\in V_{G'}$ that has  both an in-arc and an out-arc in $G'$, there exists a closed alternating walk of odd length in $G'$ which starts and ends at $v$.
		\end{enumerate}
	\end{enumerate}
	\item If $G$ is a graph:
	\begin{itemize}
		\item A subgraph $G'$ of $G$ is a mixed component of $G$ if and only if $G'$ is a component of $G$. 
		\item Note that each component of $G$ is counted twice in the total number of mixed components.  Hence, the number of mixed components of $G$ is  twice the number of components of $G$.
	\end{itemize}
\end{enumerate}	
\begin{observation}\label{observation-ASM}\normalfont
	Let $G$ be a mixed graph. Then the following properties hold:
	\begin{enumerate}[(i)]
		\item  Let $G_1,G_2,\ldots,G_k$ be the components of $G_u$. If $H$ is a mixed component of $G$, then $H_u=\underset{i\in I}{\cup} G_i$, where $I\subseteq\{1,2,\ldots,k\}$. Otherwise, the maximality condition of a special submixed graph  will be violated for Type~III mixed component. For Type~I and Type~II mixed components, it is straightforward.
		\item Each vertex of $G$ belongs to either one or two mixed components of $G$.
		\item A vertex $v$ of $G$ is part of exactly one mixed component if and only if there exists a closed alternating walk in $G$ that starts and ends at $v$ and contains an odd number of arcs.
		\item Each arc of $G$ is contained in exactly one mixed component.
		\item Each edge of $G$ belongs to either one or two mixed components.
		\item An edge $e$ of $G$ is part of exactly one mixed component if and only if there exists a closed alternating walk in $G$ that starts and ends at an endpoint $v$ of $e$ and contains an odd number of arcs.
	\end{enumerate}
\end{observation}
Let $\mathcal{S}(G)$ denote the multiset of all mixed components of a mixed graph $G$. If a component of $G$ is a graph, then it can be viewed as a mixed component of Type~I or Type~II. Therefore, we take two copies of this component in $\mathcal{S}(G)$, and designating one as a mixed component of Type~I and the other as a mixed component of Type~II.
\begin{thm}\label{ASMthm}
	If $G$ is a mixed graph, then there exists a one-to-one correspondence between $\mathcal{S}(G)$ and the set of all components of $G^A$.
\end{thm}
\begin{proof}
	Let  $V_G=\{v_1,v_2,\ldots,v_n\}$. Let $G_1,G_2,\ldots,G_k$ be the components of $G_u$. Let $G'_1,G'_2,\ldots,G'_k$ (resp. $G''_1,G''_2,\ldots,G''_k$) be the  components of the subgraph spanned by the vertex subset $\{v'_1,v'_2,\ldots,v'_n\}$ (resp. $\{v''_1,v''_2,\ldots,v''_n\}$) of $G^A$. Let $\mathcal{B}$ be the set of all components of $G^A$. We define a mapping $f$ from $\mathcal{S}(G)$ to $\mathcal{B}$ as follows. 
	
	Consider a mixed component $H$ of $G$.
	
	\textit{Case~1.} If $H$ is of Type~I, then $H=G_i$ for some $i=1,2,\ldots,k$. So, $G'_i$ is a component of $G^A$ and we associate $H$ with $G'_i$ under $f$.
	
	\textit{Case~2.} If $H$ is of Type~II, then $H=G_i$ for some $i=1,2,\ldots,k$. So, $G''_i$ is a component of $G^A$ and we associate $H$ with $G''_i$ under $f$.
	
	\textit{Case~3.} If $H$ is of Type~III, then $H$ is associated with a component $H'$ of $G^A$, which is constructed as follows: Let $G_i$ be a component of $H_u$ for some $i\in\{1,2,\ldots,k\}$.
	
	(a) Take $G'_i$, if at least one vertex of $G_i$ has non-zero out-degree in $H$.
	
	(b) Take $G''_i$, if at least one vertex of $G_i$ has non-zero in-degree in $H$.
	
	(c) For each arc $v_j\rightarrow v_k$ in $H$, join $v'_j$ and $v''_k$ by an edge.
	
	Take $H'$ to be the subgraph of $G^A$ obtained by combining one copy of each $G'_i$ and $G''_i$ mentioned above in (a) and (b), respectively along with the set of all edges formed in (c).
	
	\textbf{Claim~1:} $H'$ is a component of $G^A$.
	
	First we show that $H'$ is connected. Consider two distinct vertices $u,v\in V_{H'}$. The following cases arise:
	
	\textit{Case~i.}  $u,v\in V_{G'_r}$ for some $r\in\{1,2,\ldots,k\}$.\\
	Clearly a path exists with in $H'$ joining $u$ and $v$.
	
	\textit{Case~ii.} $u\in V_{G'_r}$ and $v\in V_{G'_l}$ for $r\neq l$.\\
	Here, $u=v'_i$ and $v=v'_j$ for some $i,j\in\{1,2,\ldots,n\}$ and $i\neq j$. We have $v_i,v_j\in V_H$. Since $H$ is a mixed component of $G$ and $v_i\in V_{G_r}$ and $v_j\in V_{G_l}$ for $r\neq l$, there exists an alternating walk joining $v_i$ and $v_j$ in $H$. This implies that at least one of the pair $(v'_i,v'_j)$, $(v'_i,v''_j)$, $(v''_i,v'_j)$ or $(v''_i,v''_j)$ of vertices is joined by a walk in $H'$. This assures that there exists a path joins any one of the above said pairs of vertices of $H'$. If there is a path joining $v'_i$ and $v'_j$ in $H'$, then the case is resolved. If there is a path joining $v'_i$ and $v''_j$ in $H'$, then both $G'_l$ and $G''_l$ are subgraphs of $H'$. This implies that at least one vertex of $G_l$ has non-zero in-degree, and one vertex of $G_l$ has non-zero out-degree in $H$. Since $H$ is a mixed component of $G$, there exists a closed alternating walk in $H$ which stars and ends at $v_j$ containing odd number of arcs. So, corresponding to this walk, we get a walk in $H'$ joining $v'_j$ and $v''_j$. Therefore, there is a path in $H'$ joining $v'_j$ and $v''_j$. Combining together, we get a path joining $v'_i$ and $v'_j$ in $H'$. For the remaining two possibilities also, in similar manner, we can show that there exists a path joining $v'_i$ and $v'_j$ in $H'$.
	
	\textit{Case~iii.}  $u=v''_i$, $v=v''_j$.\\
	As in Cases~i and ii, it can be similarly shown that there exists a path joining $u$ and $v$ in $H'$.
	
	\textit{Case~iv.}  $u\in V_{G'_r}$ and $v\in V_{G''_r}$ for some $r\in\{1,2,\ldots,k\}$.\\
	Here, $u=v'_i$ and $v=v''_j$ for some $i,j\in\{1,2,\ldots,n\}$. This implies that at least one vertex of $G_r$ has non-zero in-degree and one vertex of $G_r$ has non-zero out-degree in $H$. By definition of mixed component, there exists a closed alternating walk in $H$ which starts and ends at each vertex of $G_r$ containing odd number of arcs. In particular, at $v_i$, there is a closed alternating walk in $H$ containing odd number of arcs. This implies that there is a path joining $v'_i$ and $v''_i$ in $H'$. If $i=j$, then the case is resolved. If $i\neq j$, then there is a path joining $v''_i$ and $v''_j$ in $H'$, since $v''_i,v''_j\in V_{G''_r}$. Combining together, there exists a path joining $u=v'_i$ and $v=v''_j$ in $H'$.
	
	\textit{Case~v.} $u\in V_{G'_r}$ and $v\in V_{G''_l}$ for $r\neq l$.\\
	Here, $u=v'_i$ and $v=v''_j$ for some $i,j\in\{1,2,\ldots,n\}$ and $i\neq j$.	Correspondingly, we have $v_i,v_j\in V_H$, and since $H$ is a mixed component of $G$, there exists an alternating walk joining $v_i$ and $v_j$ in $H$. This implies that one of the pair $(v'_i,v'_j)$, $(v'_i,v''_j)$, $(v''_i,v'_j)$ or $(v''_i,v''_j)$ of vertices are joined by a path in $H'$. If there is a path joining $v'_i$ and $v''_j$ in $H'$, then the case is resolved. For the remaining three possibilities, by using a similar method described in Case~ii,  we can show that there exists a path joining $v'_i$ and $v''_j$ in $H'$.
	
	Now, the maximality property of $H$ in $G$ assures the maximality of the connected subgraph $H'$ in $G^A$. Hence, $H'$ is a component of $G^A$. Proof of Claim 1 is completed.
	
	\textbf{Claim~2:} $f$ is bijection.
	
	As mentioned in Cases~1, 2 and 3, different mixed components of $G$ are associated to different components of $G^A$ under $f$. Therefore, the association $f$ is one to one.
	To establish that $f$ is onto, we have to show that for each component $C$ of $G^A$, there exists a mixed component $C^*$ of $G$ such that $f(C^*)=C$.
	
	Let $C$ be a component of $G^A$. As per the construction of $G^A$, one of the following cases must hold: $C=G'_i$ for some $i=1,2,\ldots,k$; $C=G''_i$ for some $i=1,2,\ldots,k$; or $C$ is the union of some $G'_i$, $G''_j$ along with all the edges among them in $G^A$. We analyze each of these cases in detail below.
	
	\textit{Case~a.} If $C=G'_i$ for some $i=1,2,\ldots,k$, then we take $C^*$ as the corresponding component $G_i$ of $G_u$. Then all the vertices of $C^*$ have zero out-degree in $G$, and so it is a mixed component of Type~I. By Case~1, it follows that $f(C^*)=C$.
	
	\textit{Case~b.} If $C=G''_i$ for some $i=1,2,\ldots,k$, then we take $C^*$ as the corresponding component $G_i$ of $G_u$. Since all the vertices of $C^*$ have zero in-degree in $G$ and so it is a mixed component of Type~II. By Case~2, it follows that $f(C^*)=C$.
	
	\textit{Case~c.} If $C$ is the union of some $G'_i$, $G''_j$, and all the edges among them in $G^A$, then the following  procedure is applied:
	
	\begin{itemize}
		\item We take the union of all $G_i$s and $G_j$s corresponding to $G'_i$s and $G''_j$s in $C$, respectively.
		\item For each edge $v'_iv''_j$ in $C$, we join the corresponding vertices $v_i$ and $v_j$ with an arc directed from $v_i$ to $v_j$.
	\end{itemize}
	Let $C^*$ denotes the submixed graph of $G$ obtained from Case~c described earlier. 
	
	\textbf{Subclaim~2.1:} $C^*$, as obtained in Case~c, is a mixed component of $G$. 
	
	To prove this, consider any two vertices $v_i$ and $v_j$ in $C^*$. At least one of $v'_i$ or $v''_i$ and one of $v'_j$ or $v''_j$ must be in $C$. Since $C$ is a component of $G^A$, there is a path joining at least one of the pair $(v'_i,v'_j)$, $(v'_i,v''_j)$, $(v''_i,v'_j)$, or $(v''_i,v''_j)$ of vertices in $C$. From the construction of $C^*$, this implies the corresponding walk or alternating walk joining $v_i$ and $v_j$ in $C^*$.
	
	Now, consider $u,v\in V_{G_i}$  such that $d^+(u)\neq 0$ and $d^-(v)\neq 0$ in $G$. Then $G'_i$ and $G''_i$ are subgraphs of $G^A$. Let $v_j$ be a vertex of $G_i$. Then $v'_j$ and $v''_j$ are in $C$. Since $C$ is connected, there is a path joining $v'_j$ and $v''_j$ in $C$. This implies that there is a closed alternating walk starts and ends at $v_j$ in $C^*$ containing an odd number of arcs.
	
	Finally, since $C$ is a component of $G^A$, it assures the maximality of $C^*$ in $G$. Hence $C^*$ is a mixed component of $G$. This completes the proof of Subclaim~2.1. 
	
	By Case~3, it follows that $f(C^*)=C$, completing the proof of Claim~2, and subsequently the entire proof.
\end{proof}	

For each mixed component $H$ of $G$, we uniquely associate a  component of $G^A$ in this proof. Moving forward we refer to this component as the \textit{associated component of} $H$.

The following result is a direct consequence of Theorem~\ref{ASMthm}.
\begin{cor}\label{ASMscor}
	The number of mixed components of a mixed graph $G$ is equal to the number of components of $G^A$.
\end{cor}
\begin{defn}\normalfont
	A mixed graph $G$ is said to be \textit{$r$-regular} if $d^+(u)=d^-(u)$, and $d(u)+d^+(u)=r$ for all $u\in V_G$.
\end{defn}
It is evident that $G$ is $r$-regular if and only if $r$ is an eigenvalue of $\mathcal{I}(G)$ with corresponding eigenvector $\boldsymbol{1}_{2n}$.
\begin{lemma}\label{regular}
	Let $G$ be a mixed graph. Then $G$ is $r$-regular if and only if $G^A$ is $r$-regular.
\end{lemma}
\begin{proof}
	Let $G$ be an $r$-regular mixed graph with $V_G=\{v_1,v_2,\ldots,v_n\}$. Then for $i=1,2,\ldots,n$, we have $d(v_i)+d^+(v_i)=r=d(v_i)+d^-(v_i)$. Observe that for $i=1,2,\ldots,n$, $d(v'_i)=d(v_i)+d^+(v_i)$ and $d(v''_i)=d(v_i)+d^-(v_i)$. Hence $G^A$ is $r$-regular. By retracing the above steps, the converse can also be established.
\end{proof}
\begin{thm}\label{r-regular adjacency}
	If $G$ is an $r$-regular simple mixed graph, then $\bm{\lambda}_1(G)=r$, where multiplicity of $r$ equals the number of mixed components of $G$.
\end{thm}
\begin{proof}
	By Lemma~\ref{regular}, $G^A$ is an $r$-regular simple graph. Since $\mathcal{I}(G)$ and $A(G^A)$ have the same spectrum, the result follows in view of Corollary~\ref{ASMscor} and \cite[Theorem~3.3.1]{cvetkovic}: ``If $G$ is an $r$-regular simple graph, then $\lambda_1(G)=r$ whose multiplicity is equal to the number of components of $G$''.
\end{proof}
\begin{defn}\normalfont
	A mixed graph $G$ is said to be \textit{uniconnected} if $G$ has exactly one mixed component.
\end{defn}

\begin{example}\normalfont\label{exampuniconneced}
	The mixed graph $G$ shown in Figure~\ref{uniconnectedexamp} is uniconnected.
	\begin{figure}[ht]
		\begin{center}
			\includegraphics[scale=1]{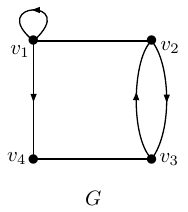}
		\end{center}\caption{An uniconnected mixed graph $G$}\label{uniconnectedexamp}
	\end{figure}
\end{example}
\begin{lemma}\label{uniconnected connected}
	A mixed graph $G$ is uniconnected if and only if $G^A$ is a connected graph.
\end{lemma}
\begin{proof}
	Follows directly from Corollary~\ref{ASMscor}.
\end{proof}
\begin{thm}\label{alternating cycle spectrum}
	Let $G$ be a simple uniconnected mixed graph on $n$ vertices. Then the following holds.
	\begin{enumerate}[(i)]
		\item If $G$ has an alternating path of length $2n-1$ that includes all vertices, all arcs, and each edge of $G$ exactly twice, then the $\mathcal{I}$-spectrum of $G$ is $2\cos\frac{\pi k}{2n+1}$ for $k=1,2,\ldots,2n$.
		\item If $G$ has an alternating cycle of length $2n-1$  that includes all vertices, all arcs, and each edge of $G$ exactly twice, then the $\mathcal{I}$-spectrum of $G$ is $2\cos\frac{\pi k}{2n+1}$ for $k=1,2,\ldots,2n$.
		\item If $G$ has an alternating cycle of length $2n$  that includes all vertices, all arcs, and each edge of $G$ exactly twice, then the $\mathcal{I}$-spectrum of $G$ is $2\cos\frac{\pi k}{n}$ for $k=1,2,\ldots,2n$.
	\end{enumerate}
\end{thm}
\begin{proof}
	Since $G$ is uniconnected, Lemma~\ref{uniconnected connected} implies that $G^A$ is connected.
	
	Under the conditions given in part~(i), $G^A$ has a path containing all its edges, which implies $G^A=P_{2n}$. Therefore, the spectrum of $A(G^A)$ is $2\cos\frac{\pi k}{2n+1}$ for $k=1,2,\ldots,2n$.
	
	Consider the conditions given in part~(ii). Since the length of the alternating cycle is $2n-1$ (an odd number) $G$ contains an odd number of arcs. This implies that $G^A=P_{2n}$ and so the spectrum of $A(G^A)$ is $2\cos\frac{\pi k}{2n+1}$ for $k=1,2,\ldots,2n$.
	
	Under the conditions given in part~(iii), $G^A$ has a cycle containing all its edges, i.e., $G^A=C_{2n}$. Consequently, the spectrum of $A(G^A)$ is $2\cos\frac{\pi k}{n}$ for $k=1,2,\ldots,2n$.
	
	Since the spectrum of $\mathcal{I}(G)$ and $A(G^A)$ are the same, the results follow.
\end{proof}

\begin{thm}\label{simple eigenvalue}
	A simple mixed graph $G$ is uniconnected if and only if $\bm{\lambda}_1(G)$ is a simple eigenvalue with a positive eigenvector.
\end{thm}
\begin{proof}
	Since $\mathcal{I}(G)=A(G^A)$, the proof follows in view of Lemma~\ref{uniconnected connected} and \cite[Corollary~1.3.8]{cvetkovic}: ``A simple graph $G$ is connected if and only if $\lambda_1(G)$ is a simple eigenvalue with a positive eigenvector''. 
\end{proof}


\section{Determinant}\label{S7 determinant}

For a mixed component $H$ of a mixed graph $G$, we define the following notations:

$\Theta_H:=\{v\in V_H\mid \text{there exits an alternating path in }H\text{ which contains }v\text{ twice}\}$,

$\Omega_H:=\{e\in E_H\mid \text{there exits an alternating path in }H\text{ which contains }e\text{ twice}\}$,

$t(H):=|E_H|+|\Omega_H|+|\vec{E}_H|$.
\begin{lemma}\label{vertexedgecount}
	Let $G$ be a mixed graph, and $H$ a mixed component of $G$. Then the associated component of $H$ in $G^A$ has $|V_H|+|\Theta_H|$ vertices, and $t(H)$ edges.
\end{lemma}
\begin{proof}
	Let $H'$ denote the associated component of $H$ in $G^A$. First, we determine the number of vertices in $H'$. It is obvious that each vertex of $V_H$ corresponds to at least one vertex and at most two vertices in $H'$. 
	Now, we claim that $v\in \Theta_H$ if and only if $H'$ has two vertices corresponding to $v$.
	Let $v\in V_H$. Suppose $H'$ has two vertices corresponding to $v$. According to the proof of Theorem~\ref{ASMthm}, the component of $H_u$ containing $v$ includes a vertex with non-zero in-degree, and a vertex with non-zero out-degree in $H$. As per the definition of mixed component, there exists an alternating walk in $H$ starting and ending at $v$ and containing odd number of arcs. This implies that there exists an alternating path in $H$ starting and ending at $v$ and containing odd number of arcs. Therefore, we have $v\in \Theta_H$. By retracing the above we can prove the converse. Thus there are $2|\Theta_H|+|V_H\setminus \Theta_H|$ vertices in $H'$. Since $\Theta_H\subseteq V_H$, we have $|V_H\setminus \Theta_H|=|V_H|-|\Theta_H|$. So, the total number of vertices of $H'$ is $|V_H|+|\Theta_H|$.
	
	Next, we determine the number of edges in $H'$. Each arc in $H$ corresponds to exactly one edge in $H'$. Each edge in $H$ corresponds to at least one edge and at most two edges in $H'$.
	Now, we claim that $e\in \Omega_H$ if and only if $H'$ has two edges corresponding to $e$.
	Let $e\in \Omega_H$. Suppose $H'$ has two edges corresponding to $e$. According to the proof of Theorem~\ref{ASMthm}, the component of $H_u$ containing $e$ includes a vertex with non-zero in-degree, and a vertex with non-zero out-degree in $H$. As per the definition of mixed component, $H$ has an alternating walk (in particular, an alternating path) starting and ending at an end vertex of $e$ and containing odd number of arcs. Suppose the above path does not contain the edge $e$ twice, we can extend it by taking $e$ as the first and last edge of the path.
	Therefore, we have $e\in \Omega_H$. By retracing the above we can prove the converse.
	Thus there are $2|\Omega_H|+|E_H\setminus \Omega_H|+|\vec{E}_H|$ edges in $H'$. Since $\Omega_H\subseteq E_H$, we have $|E_H\setminus \Omega_H|=|E_H|-|\Omega_H|$. So, the total number of edges in $H'$ is $t(H)=|E_H|+|\Omega_H|+|\vec{E}_H|$.
\end{proof}
\begin{defn}\normalfont\label{APdefinition}
	Let $G$ be a mixed graph, and let $H$ be a mixed component of $G$.
	We say $H$ satisfies the \textit{associated path property} (shortly \textit{AP property}) if $H$ is either a path, or $H$ has an alternating path which contains 
	(i) each vertex $v$ of $H$ exactly twice if $v\in \Theta_H$; otherwise, exactly once,
	(ii) all the arcs of $H$,  and
	(iii) each edge $e$ of $H$ exactly twice if $e\in \Omega_H$; otherwise, exactly once. 
\end{defn}

\begin{example}\normalfont\label{exampAPproperty}
	Consider the mixed graph $G$ shown in Figure~\ref{AP property}. The mixed graph $H$, also depicted in Figure~\ref{AP property}, is a mixed component of $G$ having the AP property. It includes an alternating path: $v_1,\{v_1,v_2\},v_2,(v_2,v_3),v_3,$ $(v_3,v_3),v_3,(v_3,v_4),v_4,(v_5,v_4),v_{5},(v_{5},v_1),v_1,\{v_1,v_2\},v_2$. Here, $\Theta_H=\{v_1,v_2,v_3\}$ and $\Omega_H=\{\{v_1,v_2\}\}$.
	\begin{figure}[ht]
		\begin{center}
			\includegraphics[scale=1]{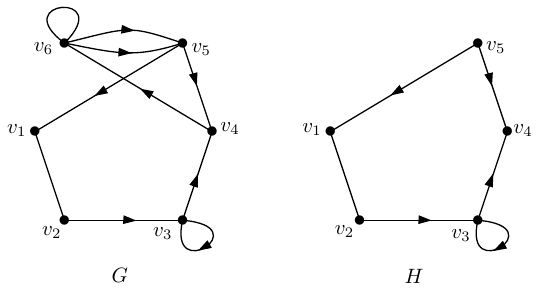}
		\end{center}\caption{The mixed graph $G$ and it's mixed component $H$ having the AP property}\label{AP property}
	\end{figure}
\end{example}

Let $G$ be a mixed graph, and $H$ be a mixed component of $G$ having the AP property. Notice that if $H$ is a path, then it's length is $t(H)$; otherwise, $H$ has an alternating path as described in Definition~\ref{APdefinition} whose length is also $t(H)$.

\begin{lemma}\label{APpath}
	Let $G$ be a mixed graph. A mixed component of $G$ has the AP property if and only if it's associated component in $G^A$ is a path.
\end{lemma}
\begin{proof}
	Let $H$ be a mixed component of $G$, and let $H'$ be the associated component of $H$ in $G^A$. Suppose $H$ has the AP property. If $H$ is a path, then $H'$ is a path. Suppose $H$ has an alternating path $P$ as described in Definition~\ref{APdefinition}. Then $P$ contains: the vertices in $\Theta_H$ exactly twice, the vertices in $V_H\setminus \Theta_H$ exactly once, the arcs in $\vec{E}_H$ exactly once, the edges in $\Omega_H$ exactly twice, and the edges in $E_H\setminus \Omega_H$ exactly once. Let $P'$ be the corresponding path of $P$ in $H'$. Then $P'$ contains exactly $2|\Theta_H|+|V_H\setminus \Theta_H|=|V_H|+|\Theta_H|$ vertices, and $2|\Omega_H|+|E_H\setminus \Omega_H|+|\vec{E}_H|=t(H)$ edges of $H'$. According to Lemma~\ref{vertexedgecount}, $P'$ contains all the vertices and edges of $H'$, proving that $H'$ is a path. 
	
	Conversely, suppose that $H'$ is a path.  Lemma~\ref{vertexedgecount} ensures $H'$ contains $|V_H|+|\Theta_H|$ vertices, and $t(H)$ edges.
	Let $P$ be the sequence of vertices, edges, and arcs in $H$ corresponding to the path $H'$. Then $P$ may be a path or an alternating path in $H$. Suppose $P$ is a path. Then $\vec{E}_H=\emptyset$ and $|\Omega_H|=\emptyset$, and so $H$ is a path of length $|E_H|$. Suppose $P$ is an alternating path. Then $P$ contains $|V_H|+|\Theta_H|$ vertices, $|E_H|+|\Omega_H|$ edges, and $|\vec{E}_H|$ arcs. As per the definition of alternating path, the above is possible only when $H$ has the AP property. 
\end{proof}	

\begin{defn}\normalfont\label{ACdefinition}
	Let $G$ be a mixed graph, and let $H$ be a mixed component of $G$. 
	We say $H$ satisfies the \textit{associated cycle property} (shortly AC property) if $H$ is either a cycle, or $H$ has an alternating cycle having an even number of arcs with starting vertex $u$ which contains 
	(i) $u$ thrice if $u\in \Theta_H$; otherwise, twice,
	(ii) each vertex $v\neq u$ of $H$ exactly twice if $v\in \Theta_H$; otherwise, exactly once, 
	(iii) all the arcs of $H$,  and
	(iv) each edge $e$ of $H$ exactly twice if $e\in \Omega_H$; otherwise, exactly once. 
\end{defn}
\begin{example}\normalfont\label{exampACproperty}
	Consider the mixed graph $G$ shown in Figure~\ref{AC property}. The mixed graph $H$, also depicted in Figure~\ref{AC property}, is a mixed component of $G$ having the AC property. It includes an alternating cycle:  $v_1,\{v_1,v_2\},v_2,(v_2,v_3),v_3,$ $(v_3,v_3),v_3,(v_3,v_4),v_4,(v_5,v_4),v_{5},(v_{5},v_1),v_1,\{v_1,v_2\},$ $v_2,(v_1,v_2),v_1$ of length $8$. Here, $\Theta_H=\{v_1,v_2,v_3\}$ and $\Omega_H=\{\{v_1,v_2\}\}$.
	\begin{figure}[ht]
		\begin{center}
			\includegraphics[scale=1]{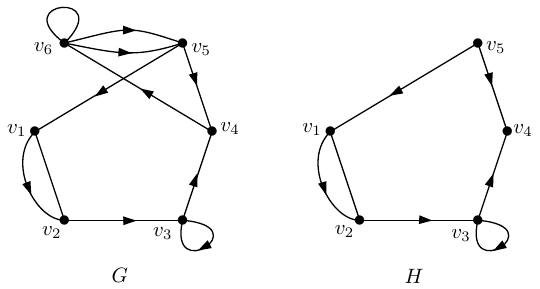}
		\end{center}\caption{The mixed graph $G$ and it's mixed component $H$ having the AC property}\label{AC property}
	\end{figure}
\end{example}
Let $G$ be a mixed graph, and $H$ be a mixed component of $G$ having the AC property. Notice that if $H$ is a cycle, then it's length is $t(H)$; otherwise, $H$ has an alternating cycle as described in Definition~\ref{ACdefinition} whose length is also $t(H)$.
\begin{lemma}\label{ACcycle}
	Let $G$ be a mixed graph. A mixed component of $G$ has the AC property if and only if it's associated component in $G^A$ is a cycle.
\end{lemma}
\begin{proof}
	Let $H$ be a mixed component of $G$, and let $H'$ be the associated component of $H$ in $G^A$. Assume that $H$ has the AC property. If $H$ is a cycle, then $H'$ is a cycle. Suppose $H$ has an alternating cycle $C$ as described in Definition~\ref{ACdefinition}. Then $C$ contains the vertices in $\Theta_H$ exactly twice (excluding the repeat of the starting vertex in the end), the vertices in $V_H\setminus \Theta_H$ exactly once (excluding the repeat of the starting vertex in the end), the arcs in $\vec{E}_H$ exactly once, the edges in $\Omega_H$ exactly twice, and the edges in $E_H\setminus \Omega_H$ exactly once. Let $C'$ be the associated cycle of $C$ in $H'$. Then $C'$ contains exactly $2|\Theta_H|+|V_H\setminus \Theta_H|=|V_H|+|\Theta_H|$ vertices (excluding the repeat of the starting vertex in the end) and $2|\Omega_H|+|E_H\setminus \Omega_H|+|\vec{E}_H|=t(H)$ edges of $H'$. In view of Lemma~\ref{vertexedgecount}, we conclude that $C'$ contains all the vertices and edges of $H'$, thus making $H'$ a cycle. 
	
	Conversely, we assume that $H'$ is a cycle. According to Lemma~\ref{vertexedgecount}, $H'$ contains $|V_H|+|\Theta_H|$ vertices and $t(H)$ edges.
	Let $C$ be the sequence of vertices, edges and arcs in $H$ corresponding to the cycle $H'$. Then $C$ could be either a cycle  or an alternating cycle having even number of arcs in $H$. If $C$ is a cycle, then $\vec{E}_H=\emptyset$ and $|\Omega_H|=\emptyset$, which means $H$ must be a cycle of length $|E_H|$. If $C$ is an alternating cycle, then it contains $|V_H|+|\Theta_H|$ vertices (excluding the repeat of the starting vertex in the end), $|E_H|+|\Omega_H|$ edges, and $|\vec{E}_H|$ arcs. As per the definition of alternating cycle, this is possible only when $H$ has the AC property. 
\end{proof}

A subgraph $H$ of a graph $G$ is called an \textit{elementary subgraph} if each component of $H$ is either an edge or a cycle. Denote by $c(H)$ and $c_1(H)$, the number of components in $H$ which are cycles and edges, respectively.

\begin{thm}\label{determinant thm}
	Let $G$ be a mixed graph in which each mixed component  either possesses the AP property or the AC property. Suppose $G$ has $q$ mixed components, $F_1,F_2,\ldots,F_q$, each with the AP property. Additionally $G$ has $l$ mixed components, $H_1,H_2,\ldots,H_l$, each possessing the AC property, such that $t(H_i)$ is even for $i=1,2,\ldots,l$. Then 
	\begin{align}\label{determinant formula}
		\det \mathcal{I}(G)=\begin{cases}
			2^{p(G)}(-1)^{2n-p(G)-\underset{i=1}{\overset{q}{\sum}}\frac{t(F_i)+1}{2}}\times \underset{S\subseteq N_l}{\sum}(-1)^{|S|-\underset{i\in S}{\sum}\frac{t(H_i)}{2}}, & \textnormal{if }t(F_i)~\textnormal{is odd for }i=1,2,\ldots,q;\\
			0, & \textnormal{otherwise}.
		\end{cases}
	\end{align}
	where 
	(i) $p(G)$ denotes the number of mixed components of $G$ that have the AC property;
	(ii) $N_l=\{1,2,\ldots,l\}$.
\end{thm}
\begin{proof}
	Since $\mathcal{I}(G)=A(G^A)$, we have $\det \mathcal{I}(G)=\det A(G^A)$. In view of \cite[Theorem~3.8]{bapatbook}: ``Let $G$ be a simple graph. Then 
	\begin{equation}\label{determinant}
		\det A(G)=\sum(-1)^{n-c_1(H)-c(H)}2^{c(H)},
	\end{equation} 
	where the summation is over all spanning elementary subgraphs $H$ of $G$'', to find $\det A(G^A)$, first we have to determine all the spanning elementary subgraphs of $G^A$. 
	
	Let $\bm F_i$ denotes the associated component of $F_i$ in $G^A$ for $i=1,2,\ldots,q$. According to Lemma~\ref{APpath}, these associated components are paths. From Lemma~\ref{vertexedgecount}, the length of $\bm F_i$ is $t(F_i)$ for $i=1,2,\ldots,q$. If at least one $t(F_i)$ is even for $i=1,2,\ldots,q$, then $G^A$ has no spanning elementary subgraph. By \eqref{determinant}, this implies that $\det A(G^A)=0$. If each $t(F_i)$ is odd for $i=1,2,\ldots,q$, then each spanning elementary subgraph of $G^A$ has exactly $\frac{t(F_i)+1}{2}$ edges from $\bm F_i$. This implies that each elementary spanning subgraph of $G^A$ has totally $\sum_{i=1}^{q}\frac{t(F_i)+1}{2}$ edges from $\underset{i=1}{\overset{q}{\cup}}\bm F_i$.
	
	According to Lemma~\ref{ACcycle}, the number of mixed components of $G$ having the AC property is equal to the number of components of $G^A$ that are cycles. Therefore, $G^A$ has $p(G)$ components that are cycles. Let $\bm H_i$ denotes the associated component of $H_i$ in $G^A$ for $i=1,2,\ldots,l$. From Lemma~\ref{vertexedgecount}, it follows that $\bm H_i$ is an even cycle of length $t(H_i)$, while the remaining $p(G)-l$ cycles in $G^A$ are of odd length. All the odd cycles of $G^A$ appear unchanged in each spanning elementary subgraph of $G^A$.
	
	At this point, we have established that the spanning elementary subgraphs of $G^A$ have $\sum_{i=1}^{q}\frac{t(F_i)+1}{2}$ edges and $p(G)-l$ cycles in common. The variation among these spanning elementary subgraphs depends on the components $\bm H_1,\bm H_2,\ldots,\bm H_l$ of $G^A$.
	
	For each $i=1,2,\ldots,l$, since $\bm H_i$ is a cycle of even length, it can either appear as it is or as two different sets of $\frac{t(H_i)}{2}$ independent edges in a spanning elementary subgraph of $G^A$. 
	
	Thus, for each $S\subseteq N_l$, there are $2^{|S|}$ spanning elementary subgraphs of $G^A$, each containing $\sum_{i=1}^{q}\frac{t(F_i)+1}{2}+\underset{i\in S}{\sum}\frac{t(H_i)}{2}$ edges and $p(G)-|S|$ cycles. This leads to the following contribution for each subset $S\subseteq N_l$ in the summation given in \eqref{determinant}
	$$2^{|S|}\times(-1)^{2n-p(G)+|S|-\sum_{i=1}^{q}\frac{t(F_i)+1}{2}-\underset{i\in S}{\sum}\frac{t(H_i)}{2}}2^{p(G)-|S|}.$$ Hence the result follows.
\end{proof}
\begin{example}\normalfont
	Consider the mixed graph $G$ shown in  Figure~\ref{determinantexample}. The mixed graphs $H_1,H_2,G_1,$ $G_2,F_1,F_2$, and $F_3$ shown in  Figure~\ref{determinantexample}, are the mixed components of $G$. Among these $F_1$, $F_2$ and $F_3$ have the AP property, while $H_1$, $H_2$, $G_1$, and $G_2$ have the AC property. Here $n=12$, $p(G)=4$, $q=3$, $l=2$, $t(F_1)=3$, $t(F_2)=1$, $t(F_3)=1$, $t(H_1)=6$ and $t(H_2)=4$. Thus from \eqref{determinant formula}, $\det \mathcal{I}(G)=2^4(-1)^{2\times10-(\frac{3+1}{2}+\frac{1+1}{2}+\frac{1+1}{2})}\times((-1)^0+(-1)^{1-\frac{6}{2}}+(-1)^{1-\frac{4}{2}}+(-1)^{2-\frac{6}{2}-\frac{4}{2}})=2^4(1+1-1-1)=0$. This fact is also evident since $\mathcal{I}(G)$ has two identical columns.
	\begin{figure}[ht]
		\begin{center}
			\includegraphics[scale=1]{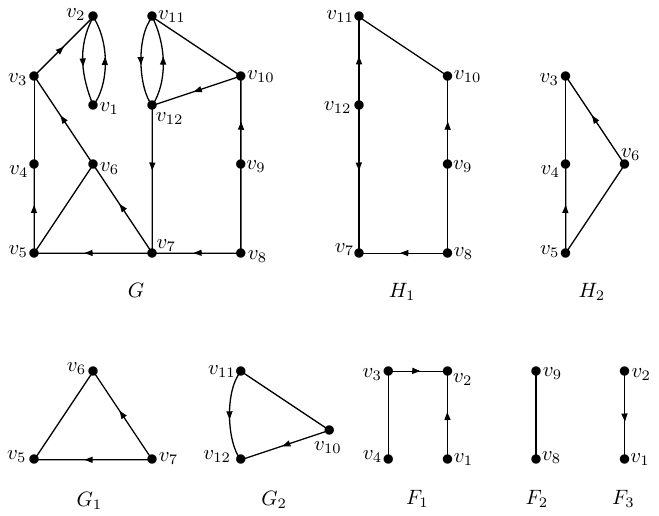}
		\end{center}\caption{The mixed graph $G$; the mixed components $H_1$, $H_2$, $G_1$, and $G_2$ of $G$ having the AC property; the mixed components $F_1$, $F_2$, and $F_3$ of $G$ having the AP property.}\label{determinantexample}
	\end{figure}
\end{example}

\section{Bounds}\label{S8 bounds}
Let $G$ be a mixed graph. We define the following notations:

$\Delta_1(G):=\max \{d(u)+d^+(u)\mid u\in V_G\}$,

$\Delta_2(G):=\max \{d(u)+d^-(u)\mid u\in V_G\}$,

$\delta_1(G):=\min \{d(u)+d^+(u)\mid u\in V_G\}$,

$\delta_2(G):=\min \{d(u)+d^-(u)\mid u\in V_G\}$. 
\begin{thm}\label{thm-bound1}
	If $G$ is a simple mixed graph, then $|\lambda|\leq \max\{\Delta_1(G),\Delta_2(G)\}$ for each eigenvalue $\lambda$ of $\mathcal{I}(G)$.
\end{thm}
\begin{proof}
	Notice that $\Delta(G^A)=\max\{\Delta_1(G),\Delta_2(G)\}$. Since $\mathcal{I}(G)$ and $A(G^A)$ have the same spectrum, the result follows in view of \cite[Proposition~1.1.1]{cvetkovic}: ``Let $G$ be a simple graph. Then $|\lambda|\leq \Delta(G)$ for each eigenvalue $\lambda$ of $A(G)$''. 
\end{proof}

\begin{thm}\label{thm-bound2}
	If $G$ is a simple mixed graph, then $$\min\{\delta_1(G),\delta_2(G)\}\leq\frac{1}{2}\left(\delta_1(G)+\delta_2(G)\right)\leq\bm{\lambda}_1(G)\leq \max\{\Delta_1(G),\Delta_2(G)\}.$$
\end{thm}
\begin{proof}
	Let $G$ has $n$ vertices. Let $X$ be a column vector in $\mathbb{R}^{2n}$ with the usual Euclidean norm. Consider the extremal representation
	\begin{eqnarray*}
		\nonumber\bm{\lambda}_1(G)&=&\underset{\lVert X\rVert=1}{\max}\{X^T\mathcal{I}(G)X\}\\\nonumber
		&=&\underset{X\neq 0}{\max}\bigg\{\frac{X^T\mathcal{I}(G)X}{X^TX}\bigg\}.
	\end{eqnarray*}
	This implies that
	\begin{eqnarray}
		\nonumber
		\bm{\lambda}_1(G)&\geq&\frac{\boldsymbol{1}_{2n}^T\mathcal{I}(G)\boldsymbol{1}_{2n}}{\boldsymbol{1}_{2n}^T\boldsymbol{1}_{2n}}\\\nonumber
		&=&\frac{1}{2n}\bigg(\underset{u\in V_G}{\sum}(d(u)+d^+(u))+\underset{u\in V_G}{\sum}(d(u)+d^-(u))\bigg)\\\nonumber
		&\geq&\frac{1}{2}(\delta_1(G)+\delta_2(G))\\\nonumber
		&\geq&\min\{\delta_1(G),\delta_2(G)\}.
	\end{eqnarray}	
	By Theorem~\ref{thm-bound1}, $\bm{\lambda}_1(G)\leq \max\{\Delta_1(G),\Delta_2(G)\}$. Hence, the result follows.
\end{proof}

The largest $\mathcal{I}$-eigenvalue of a mixed graph $G$, i.e., $\bm{\lambda}_1(G)$ is referred to as the $\mathcal{I}$-\textit{spectral radius of $G$}.

\begin{thm}
	Let $G$ be a mixed graph on $n$ vertices. Then $\displaystyle\frac{1}{2}(\bm{\lambda}_1(G)+\bm{\lambda}_{2n}(G))\leq\lambda_1(G_u)$.
\end{thm}
\begin{proof}
	The proof follows by taking $M=\mathcal{I}(G)$, $P=R=A(G_u)$, and $Q=\vec{A}(G_d)$ in \cite[Proposition~1.3.16]{cvetkovic}: ``If $M$ is a symmetric $n\times n$ matrix with real entries and 
	$$M=\begin{bmatrix}
		P&Q\\
		Q^T&R
	\end{bmatrix},$$ then $\lambda_1(M)+\lambda_n(M)\leq\lambda_1(P)+\lambda_1(R)$''.
\end{proof}

\begin{thm}\label{thm-bound4}
	Let $G$ be a mixed graph on $n$ vertices. Then $\bm{\lambda}_{2n-1}(G)\leq\displaystyle\frac{1}{n}[2e(G)+2l(G)+a(G)]\leq\bm{\lambda}_1(G)$ and $\bm{\lambda}_{2n}(G)\leq\displaystyle\frac{1}{n}[2e(G)+2l(G)-a(G)]\leq\bm{\lambda}_2(G)$.
\end{thm}
\begin{proof}
	Consider the matrix $\mathcal{I}(G)$ as described in \eqref{adjmatrix block}. Here, the sums of the entries of $A(G_u)$ and $\vec{A}(G_d)$ are $2e(G)+2l(G)$ and $a(G)$, respectively. As per \cite[Corollary~1.3.13]{cvetkovic}, the eigenvalues $\displaystyle\frac{1}{n}[2e(G)+2l(G)+a(G)]$ and $\displaystyle\frac{1}{n}[2e(G)+2l(G)-a(G)]$
	of the matrix 
	$$\begin{bmatrix}
		\frac{2e(G)+2l(G)}{n} & \frac{a(G)}{n}\\
		\frac{a(G)}{n} & \frac{2e(G)+2l(G)}{n}
	\end{bmatrix}$$
	interlace the eigenvalues of $\mathcal{I}(G)$. This completes the proof.
\end{proof}

\begin{thm}\label{thm-bound5}
	Let $G$ be a simple mixed graph on $n$ vertices. Then $$\displaystyle\bm{\lambda}_1(G)\leq \left(\frac{(2n-1)(2e(G)+a(G))}{n}\right)^{\frac{1}{2}}.$$
\end{thm}
\begin{proof}
	$G^A$ is a simple graph on $2n$ vertices, and the number of edges of $G^A$ is $2e(G)+a(G)$. Since $\mathcal{I}(G)$ and $A(G^A)$ have the same spectrum, the proof follows from \cite[Theorem~3.15]{bapatbook}: ``Let $G$ be a simple graph on $n$ vertices and $m$ edges. Then $\lambda_1(G)\leq \left(\frac{2m(n-1)}{n}\right)^{\frac{1}{2}}$''.
\end{proof}
\begin{thm}\label{thm-bound6}
	Let $G$ be a simple uniconnected mixed graph. Then $\bm{\lambda}_1(G-u)<\bm{\lambda}_1(G)$ for $u\in V_G$.
\end{thm}
\begin{proof}
	According to Lemma~\ref{uniconnected connected}, $G^A$ is a simple connected graph. For $u\in V_G$, we consider the corresponding vertices $u'$ and $u''$ of $G^A$. Let $G'=G-u$ and $G_1=G^A-u'$. Then $G'^A=G_1-u''=G^A-\{u',u''\}$. From  \cite[Proposition~1.3.9]{cvetkovic}: ``Let $G$ be a simple connected graph. If $u\in V_G$, then $\lambda_1(G-u)<\lambda_1(G)$'', we have $\lambda_1(G_1)<\lambda_1(G^A)$. Since $G'^A$ is an induced subgraph of $G_1$, as per \cite[Corollary~1.3.12]{cvetkovic}: ``Let $G$ be a simple graph and let $H$ be an induced subgraph of $G$. Then the eigenvalues of $A(H)$ interlace the eigenvalues of $A(G)$'', we have $\lambda_1(G'^A)\leq \lambda_1(G_1)$. This implies that $\lambda_1(G'^A)<\lambda_1(G^A)$. Since $\bm{\lambda}_1(G)=\lambda_1(G^A)$
	and $\bm{\lambda}_1(G')=\lambda_1(G'^A)$, the result follows.
\end{proof}

\begin{thm}\label{thm-bound7}
	Let $G$ be a simple uniconnected mixed graph. Then $\bm{\lambda}_1(G-a)<\bm{\lambda}_1(G)$ for $a\in \vec{E}_G$.
\end{thm}
\begin{proof}
	In view of Lemma~\ref{uniconnected connected}, $G^A$ is a simple connected graph. For $a\in \vec{E}_G$, we consider the corresponding edge $e$ of $G^A$.
	Let $G'=G-a$. Then $G'^A=G^A-e$. Since $\bm{\lambda}_1(G)=\lambda_1(G^A)$ and $\bm{\lambda}_1(G')=\lambda_1(G'^A)$, the result follows from  \cite[Proposition~1.3.10]{cvetkovic}: ``Let $G$ be a simple connected graph. If $e\in E_G$, then $\lambda_1(G-e)<\lambda_1(G)$''. 
\end{proof}

\begin{thm}\label{ineq}
	Let $G$ be an $r$-regular simple mixed graph on $n$ vertices. If $G'$ is an induced submixed graph of $G$ with $n'<n$ vertices, then $$\bm{\lambda}_{2n}(G)\leq \displaystyle\frac{nd'-rn'}{n-n'}\leq\bm{\lambda}_2(G),$$
	where $d'=\displaystyle\frac{1}{2n'}\underset{u\in V_{G'}}{\sum}[2d(u)+d^+(u)+d^-(u)]$.
\end{thm}
\begin{proof}
	According to Lemma~\ref{regular}, $G^A$ is an $r$-regular simple graph. Clearly, $G'^A$ is an induced subgraph of $G^A$ with $2n'$ vertices, and $d'=\frac{1}{2n'}\underset{u\in V_{G'}}{\sum}[2d(u)+d^+(u)+d^-(u)]$ is the average degree of $G'^A$.  In view of \cite[Theorem~3.5.7]{cvetkovic}: ``Let $G$ be an $r$-regular simple graph with $n$ vertices. Let $G_1$ be an induced subgraph of $G$ with $n_1$ vertices and average degree $d_1$. Then $$\frac{n_1(r-\lambda_{n}(G))}{n}+\lambda_{n}(G)\leq d_1\leq \frac{n_1(r-\lambda_{2}(G))}{n}+\lambda_{2}(G)\text{'',}$$ 
	we get the inequalities 
	$$\lambda_{2n}(G^A)\leq \frac{nd'-rn'}{n-n'}\leq \lambda_{2}(G^A).$$
	Since $\mathcal{I}(G)$ and $A(G^A)$ have the same spectrum, the result follows.
\end{proof}


\begin{lemma}\label{alphalemma}
	For a mixed graph $G$, $2\alpha(G)\leq\alpha(G^A)$.
\end{lemma}
\begin{proof}
	Let $V_G=\{v_1,v_2,\ldots,v_n\}$. Without loss of generality, let $U=\{v_1,v_2,\ldots,v_k\}$ be an independent set of $G$ with $|U|=\alpha(G)$. This implies that no two vertices in $U$ are adjacent in $G$. Therefore, no two vertices in $U'=\{v_1',v_2',\ldots,v_k',v_1'',v_2'',\ldots,v_k''\}$ are adjacent in $G^A$. Thus $U'$ is an independent set of $G^A$, and $\alpha(G^A)\geq |U'|=2\alpha(G)$.
\end{proof}
\begin{thm}
	Let $G$ be a simple mixed graph on $n$ vertices. Let $n^+$ and $n^-$ denote the number of positive and negative eigenvalues of $\mathcal{I}(G)$, respectively. Then $\alpha(G)\leq \min\left\{\frac{2n-n^+}{2},\frac{2n-n^-}{2}\right\}$.
\end{thm}
\begin{proof}
	Since $\mathcal{I}(G)=A(G^A)$, we have that $A(G^A)$ has $n^+$ positive eigenvalues, and $n^-$ negative eigenvalues. Then proof follows from Lemma~\ref{alphalemma} and \cite[Theorem~3.10.1]{cvetkovic}: ``Let $G$ be a simple graph on $n$ vertices. Let $n^+$ and $n^-$ denote the number of positive and negative eigenvalues of $A(G)$, respectively. Then $\alpha(G)\leq\min \{n-n^+,n-n^-\}$''. 
\end{proof}

\begin{thm}\label{thm-bound10}
	Let $G$ be an $r$-regular simple mixed graph on $n$ vertices. Then $\displaystyle\alpha(G)\leq\frac{-n\bm{\lambda}_{2n}(G)}{r-\bm{\lambda}_{2n}(G)}.$
\end{thm}
\begin{proof}
	From Lemma~\ref{regular}, $G^A$ is an $r$-regular simple graph on $2n$ vertices. As per Theorem~\ref{r-regular adjacency}, we have $\lambda_1(G)=r$. Since  $\mathcal{I}(G)$ and $A(G^A)$ have the same spectrum, the proof follows from Lemma~\ref{alphalemma} and \cite[Theorem~3.10.2]{cvetkovic}: ``If $G$ is a simple regular graph on $n$ vertices, then  $\alpha(G)\leq\frac{-n\lambda_{n}(G)}{\lambda_1(G)-\lambda_{n}(G)}$''.
\end{proof}

\begin{defn}\normalfont
	Let $G$ be a simple mixed graph. A \textit{clique} of $G$ is a complete submixed graph of $G$. The \textit{clique number} of $G$, denoted by $\omega(G)$, is the number of vertices in a largest clique of $G$.
\end{defn}
\begin{thm}
	Let $p^-,p^0$ and $p^+$ denote the number of  $\mathcal{I}$-eigenvalues of a simple mixed graph $G$ which are less than, equal to and greater than $0$, respectively and let $q^-,q^0$ and $q^+$ denote the number of $\mathcal{I}$-eigenvalues of $G$ which are less than, equal to and greater than $-2$, respectively. Let $s=\min \left\{1+\frac{\bm{\lambda}_1(G)}{2},\frac{p^0+p^-+1}{2},p^0+p^+-1,\right.$ $\left.q^0+q^-+1,\frac{q^0+q^+}{2}\right\}$. Then $\omega(G)\leq s$.
\end{thm}
\begin{proof}
	Suppose that $G$ contains a clique $H$ on $k\geq 1$ vertices. Then $\mathcal{I}(H)$ is of the form $J_2\otimes(J_k-I_k)$.
	Since $H$ is an induced submixed graph of $G$, by Theorem~\ref{thm-interlace}, the eigenvalues of $\mathcal{I}(H)$ interlace the eigenvalues of $\mathcal{I}(G)$. Since the eigenvalues of $J_2\otimes(J_k-I_k)$ are $2(k-1)$ with multiplicity $1$, $0$ with multiplicity $k$, and $-2$ with multiplicity $k-1$, we have the following:
	\begin{eqnarray}\label{1.3}
		&\bm{\lambda}_{2n-2k+1}(G)\leq 2(k-1)\leq \bm{\lambda}_1(G);\\\label{1.4}
		&\bm{\lambda}_{2n-2k+i}(G)\leq 0\leq \bm{\lambda}_i(G)~\text{for}~i=2,\ldots,k+1;\\\label{1.5}
		&\bm{\lambda}_{2n-2k+i}(G)\leq -2\leq \bm{\lambda}_i(G)~\text{for}~i=k+2,\ldots,2k.
	\end{eqnarray}
	From the second inequality of \eqref{1.3}, we get $k\leq 1+\displaystyle\frac{\bm{\lambda}_1(G)}{2}$.
	By \eqref{1.4}, we get $k\leq \displaystyle\frac{p^0+p^-+1}{2}$ and $k\leq p^0+p^+-1$.
	From \eqref{1.5}, we get $k\leq q^0+q^-+1$ and $k\leq \displaystyle\frac{q^0+q^+}{2}$. Therefore, we obtain $k\leq s$. Since $k$ is arbitrary, we get $\omega(G)\leq s$, as desired.	
\end{proof}
\section{Concluding remarks}
The integrated adjacency matrix of a mixed graph, introduced in this paper, accommodates entries corresponding to multiple edges, multiple arcs, and multiple loops. Furthermore, any mixed graph can be uniquely determined from its integrated adjacency matrix. We have established results that reveal the interplay between the eigenvalues of the integrated adjacency matrix of a mixed graph and its structural properties.

It is evident that some of the spectral graph theoretic results for simple graphs, utilized in this paper, can be extended to graphs with loops and/or multiple edges. Consequently, Theorems~\ref{thm-interlace}, \ref{r-regular adjacency}, \ref{simple eigenvalue}, \ref{thm-bound1}, \ref{thm-bound2}, \ref{thm-bound5}, \ref{thm-bound6}, \ref{thm-bound7}, \ref{ineq}, and \ref{thm-bound10}, which were proved for simple mixed graphs, can also be extended to mixed graphs with loops, directed loops, multiple edges, and/or multiple arcs. Additionally, Theorem~\ref{entry} can be extended to loopless mixed graphs. Furthermore, in part~II of this paper~\cite{Kalaimatrices2}, we introduce and study three additional matrices for mixed graphs: integrated Laplacian matrix, integrated signless Laplacian matrix, and normalized integrated Laplacian matrix, all defined using the integrated adjacency matrix.

\end{document}